\documentclass[final]{siamltex704}

\usepackage{epsfig}                                                        
\usepackage{latexsym,amsmath, enumerate}
\usepackage{euscript}
\usepackage{graphicx}
\usepackage{rotating}
\usepackage{url}
\usepackage{graphicx,epsfig}  
\usepackage{epstopdf} 
\usepackage{color}

\newcommand{\bfA}{ {\mathbf A} }
\newcommand{\bfB}{ {\mathbf B} }
\newcommand{\bfC}{ {\mathbf C} }
\newcommand{\bfD}{ {\mathbf D} }
\newcommand{\bfE}{ {\mathbf E} }
\newcommand{\bfF}{ {\mathbf F} }
\newcommand{\bfG}{ {\mathbf G} }
\newcommand{\bfH}{ {\mathbf G} }
\newcommand{\bfI}{ {\mathbf I} }
\newcommand{\bfJ}{ {\mathbf J} }

\newcommand{\bfM}{ {\mathbf M} }
\newcommand{\bfN}{ {\mathbf N} }

\newcommand{\bfP}{ {\mathbf P} }

\newcommand{\bfR}{ {\mathbf R} }
\newcommand{\bfS}{ {\mathbf S} }
\newcommand{\bfT}{ {\mathbf T} }
\newcommand{\bfU}{ {\mathbf U} }
\newcommand{\bfV}{ {\mathbf V} }
\newcommand{\bfW}{ {\mathbf W} }
\newcommand{\bfX}{ {\mathbf X} }
\newcommand{\bfY}{ {\mathbf Y} }
\newcommand{\bfZ}{ {\mathbf Z} }
\newcommand{\bfTheta}{\mbox{\boldmath$\Theta$} }

\newcommand{\bfb}{ {\mathbf b} }

\newcommand{\bff}{ {\mathbf f} }

\newcommand{\bfq}{ {\mathbf q} }

\newcommand{\bfu}{ {\mathbf u} }
\newcommand{\bfv}{ {\mathbf v} }
\newcommand{\bfw}{ {\mathbf w} }
\newcommand{\bfx}{ {\mathbf x} }
\newcommand{\bfy}{ {\mathbf y} }
\newcommand{\bfz}{ {\mathbf z} }

\newcommand{\bfGsp}{ {\bfG}_{\rm sp}}
\newcommand{\bfGip}{ {\bfP}}

\newcommand{\cbfA}{\mbox{\boldmath${\EuScript{A}}$} }
\newcommand{\cbfE}{\mbox{\boldmath${\EuScript{E}}$} }
\newcommand{\cbfB}{\mbox{\boldmath${\EuScript{B}}$} }
\newcommand{\cbfC}{\mbox{\boldmath${\EuScript{C}}$} }
\newcommand{\cbfD}{\mbox{\boldmath${\EuScript{D}}$} }

\newcommand{\bfGr}{ {\widetilde{\bfG}} }
\newcommand{\bfEr}{ {\widetilde{\bfE}}  }
\newcommand{\bfAr}{ {\widetilde{\bfA}}  }
\newcommand{\bfBr}{ {\widetilde{\bfB}}  }
\newcommand{\bfCr}{{\widetilde{\bfC}}  }
\newcommand{\bfDr}{ {\widetilde{\bfD}} }

\newcommand{\bfErsp}{ {\bfEr_{\rm sp}} }
\newcommand{\bfArsp}{ {\bfAr_{\rm sp}} }
\newcommand{\bfBrsp}{ {\bfBr_{\rm sp}} }
\newcommand{\bfCrsp}{ {\bfCr_{\rm sp}} }
\newcommand{\bfGrsp}{ {\widetilde{\bfG}_{\rm sp}} }
\newcommand{\bfGrip}{ {\widetilde{\bfP}} }

\newtheorem{remark}{Remark}[section]
\newtheorem{example}{Example}[section]
\newtheorem{algorithm}{Algorithm}[section]

\def\Htwo{ { {\mathcal H}_2} }
\def\Hinf{ {{\mathcal H}_{\infty}} }
\newfont{\Bb}{msbm10 scaled\magstep0}
\def\IR{\mbox {\Bb R}}
\def\IC{\mbox {\Bb C}}

\newcommand{\bfsfb}{\mathbf{\mathsf{b}}}
\newcommand{\bfsfc}{\mathbf{\mathsf{c}}}

\newcommand{\bfpi}{{\boldsymbol{\Pi}}}

\title{Model Reduction of Descriptor Systems by Interpolatory Projection Methods
}

\author{Serkan Gugercin
\thanks{Serkan Gugercin is with 
the Department of Mathematics, 
Virginia Tech., 
Blacksburg, VA, 24061-0123, USA,  
e-mail: {\texttt{gugercin@math.vt.edu}}.} 
\and Tatjana Stykel
\thanks{Tatjana Stykel is with 
Institut f\"{u}r Mathematik, 
Universit\"at Augsburg, 
Universit\"atsstra{\ss}e 14, 
86159 Augsburg, Germany, 
e-mail: {\texttt{stykel@math.uni-augsburg.de}}. }     
\and 
Sarah Wyatt
\thanks{Sarah Wyatt is with 
the Department of Mathematics, 
Indian River State College, 
Fort Pierce, FL, 34981, USA,  
e-mail: {\texttt{swyatt@irsc.edu}}.} 
}

\begin{document}
\maketitle

\begin{abstract}
In this paper, we investigate interpolatory projection framework for model reduction 
of descriptor systems. With a simple numerical example, we first illustrate that employing
subspace conditions from the standard state space settings to descriptor systems generically 
leads to unbounded $\Htwo$ or $\Hinf$ errors due to the mismatch of the polynomial parts 
of the full and reduced-order transfer functions. We then develop modified interpolatory subspace 
conditions based on the deflating subspaces that guarantee a~bounded error.  For the special cases 
of index-$1$ and index-$2$ descriptor systems, we also show how to avoid computing these deflating 
subspaces explicitly while still enforcing interpolation. The question of how to choose interpolation 
points optimally naturally arises as in the standard state space setting. We answer this question in the 
framework of the \mbox{$\Htwo$-norm} by extending the Iterative Rational Krylov Algorithm (IRKA)
to descriptor systems. Several numerical examples are used to illustrate the theoretical discussion.
\end{abstract}


\pagestyle{myheadings}
\thispagestyle{plain}
\markboth{Serkan Gugercin, Tatjana Stykel, and Sarah Wyatt}{INTERPOLATORY PROJECTION METHODS for DAEs}

\section{Introduction}
We discuss interpolatory model reduction of differential-\linebreak al\-geb\-raic equations (DAEs), 
or descriptor systems, given by
\begin{equation} \label{dae_fom}
\arraycolsep 2pt
    \begin{array}{rcl}
      \bfE\, \dot{\bfx}(t) & = & \bfA\bfx(t)+\bfB\bfu(t),  \\
       \bfy(t) & = & \bfC\bfx(t) + \bfD \bfu(t),
    \end{array}  
\end{equation}
where $\bfx(t) \in \IR^n$, $\bfu(t) \in \IR^m$ and $\bfy(t) \in \IR^p$ are
 the states,  inputs and  outputs, respectively,  
$\bfE\in \IR^{n\times n}$ is a {\it singular} matrix, 
$\bfA\in \IR^{n\times n}$,  $\bfB\in \IR^{n\times m}$,  $\bfC\in  \IR^{p\times n}$, 
and $\bfD\in \IR^{p\times m}$. Taking the Laplace transformation of system (\ref{dae_fom}) 
with zero initial condition $\bfx(0) = \mathbf{0}$, we obtain 
$\widehat{\bfy}(s) = \bfG(s) \widehat{\bfu}(s)$, 
where $\widehat{\bfu}(s)$ and $\widehat{\bfy}(s)$ denote the Laplace transforms of 
$\bfu(t)$ and $\bfy(t)$, respectively, and 
$\bfG(s) = \bfC(s\bfE - \bfA)^{-1}\bfB + \bfD$ is a~transfer function of (\ref{dae_fom}).
By following the standard abuse of notation, we will denote both the dynamical system and 
its transfer function by $\bfG$.
  
Systems of the form (\ref{dae_fom}) with extremely large state space dimension $n$ arise 
in various applications such as electrical circuit simulations, multibody dynamics, 
or semidiscretized partial differential equations. Simulation and control in these large-scale 
settings is a~huge computational burden. Efficient model utilization becomes crucial where 
model reduction offers a~remedy. The goal of model reduction is to replace the original dynamics 
in (\ref{dae_fom}) by a~model of the same form but with much smaller state space dimension 
such that this reduced model is a~high fidelity approximation to the original one. Hence, 
 we seek a~reduced-order model
 \begin{equation} \label{dae_rom}
  \arraycolsep 2pt
  \begin{array}{rcl}
       \bfEr\, \dot{\widetilde{\bfx}}(t) & = & \bfAr\widetilde{\bfx}(t)+\bfBr\bfu(t),  \\
       \widetilde{\bfy}(t) & = & \bfCr\widetilde{\bfx}(t) + \bfDr \bfu(t),
    \end{array}
\end{equation}
where $\bfEr,\bfAr\in\IR^{r\times r}$, $\bfBr\in\IR^{r\times m}$, $\bfCr\in\IR^{p\times r}$, 
and $\bfDr\in \IR^{p\times m}$ such that $r \ll n$, and the error 
$\bfy - \widetilde{\bfy}$ is small with respect 
to a~specific norm over a~wide range of inputs $\bfu(t)$ with bounded energy. In the frequency domain, 
this means that the transfer function of~(\ref{dae_rom}) given by 
$\bfGr(s)=\bfCr(s\bfEr - \bfAr)^{-1}\bfBr + \bfDr$ approximates $\bfG(s)$ well, i.e., the error
$\bfG(s)-\bfGr(s)$ is small in a~certain system norm.
 
The reduced-order model (\ref{dae_rom}) can be obtained via projection as follows.
We first construct  two $n \times r$ matrices  $\bfV$ and $\bfW$, 
approximate the full-order state $\bfx(t)$ by $\bfV\widetilde{\bfx}(t)$, 
and then enforce the Petrov-Galerkin condition
$$
\mathbf{W}^{T}\left(\bfE \bfV\dot{\widetilde{\bfx}} (t)-
\bfA\bfV\widetilde{\bfx}(t)-\bfB\,\bfu(t)\right)=\mathbf{0}, \qquad
\quad \widetilde{\bfy}(t) = \bfC\bfV \widetilde{\bfx}(t) + \bfD \bfu(t).
$$
As a~result, we obtain the reduced-order model (\ref{dae_rom}) with the system matrices
\begin{equation}  \label{red_projection}
\begin{array}{ll}
\bfEr= \bfW^{T} \bfE \bfV, & \qquad\bfAr = \bfW^{T} \bfA \bfV,\\[.1in]
 \bfBr = \bfW^{T}\bfB,  & \qquad \bfCr = \bfC \bfV, \qquad \bfDr  = \bfD.
\end{array}  
\end{equation}
The projection matrices $\bfV$ and $\bfW$ determine the subspaces of interest and can be 
computed in many different ways. 

In this paper, we consider projection-based interpolatory model reduction me\-thods, 
where the choice of $\bfV$ and $\bfW$ enforces certain tangential interpolation of the original transfer function.
These methods will be presented in Section~\ref{sec:tanginterp} in more detail.
Projection-based interpolation with multiple interpolation points was initially proposed 
by Skelton~{\it et.~al.} in \cite{devillemagne1987model, yousuff1985lsa,yousouff1984cer}. 
Grimme \cite{grimme1997krylov} has later developed a~numerically efficient framework 
using the rational Krylov subspace method of Ruhe \cite{ruhe1984rational}. 
The tangential rational interpolation framework, we will be using here, is due 
to a~recent work by Gallivan {\it et al.} \cite{gallivan2005model}. 

Unfortunately, it is often assumed that extending interpolatory model reduction from 
standard state space systems with $\bfE=\bfI$ to descriptor systems with singular $\bfE$ 
is as simple as replacing $\bfI$ by $\bfE$. In Section~\ref{sec:tanginterp}, 
we present an~example showing that this naive 
approach may lead to a~poor approximation with an~unbounded error $\bfG(s)-\bfGr(s)$ although 
the classical interpolatory subspace conditions are satisfied. 
In Section~\ref{sec:int_dae}, we modify these conditions in order to enforce  
bounded error. The theo\-re\-ti\-cal result will take advantage of the spectral projectors. 
Then using the new subspace conditions, we extend in Section~\ref{sec:optinterp} 
the optimal $\Htwo$ model reduction method of \cite{gugercin2008hmr} to descriptor systems. 
Sections~\ref{sec:int_dae} and \ref{sec:optinterp} make explicit usage of deflating subspaces 
which could be numerically demanding for general problems. Thus, for the special cases of 
index-1 and index-2 descriptor systems, we show in Sections~\ref{sec:index1} and \ref{sec:index2}, 
respectively, how to apply interpolatory model reduction without explicitly computing 
 the deflating subspaces. Theoretical discussion will be supported by several numerical examples. 
In particular, in Section \ref{sec:inlet}, we present an~example, where the balanced truncation 
approach \cite{stykel2004gramian} is prone to failing 
due to problems solving the generalized Lyapunov equations, while the (optimal) interpolatory model 
reduction can be effectively applied.

\section{Model reduction by tangential rational interpolation}
\label{sec:tanginterp}

The goal of model reduction by tangential interpolation is to construct a~reduced-order model 
(\ref{dae_rom}) such that its transfer function $\bfGr(s)$ interpolates 
the original one, $\bfG(s)$, at selected points in the complex plane along selected directions. 
We will use the notation of \cite{antoulas2010imr} to define this problem more precisely: 
Given $\bfG(s)=\bfC(s\bfE - \bfA)^{-1}\bfB + \bfD$, the left interpolation points 
$\{\mu_i\}_{i=1}^q$, $\mu_i \in \IC$, 
together with the left tangential directions $\{\bfsfc_i\}_{i=1}^q$, $\bfsfc_i\in\IC^p$, and 
the right interpolation points $\{\sigma_j\}_{j=1}^r$, $\sigma_j\in\IC$, together 
with the right tangential directions $\{\bfsfb_j\}_{j=1}^r$, $\bfsfb_j\in\IC^m$, we seek 
to find a~reduced-order model $\bfGr(s)=\bfCr(s\bfEr - \bfAr)^{-1}\bfBr + \bfDr$ 
that is a~tangential interpolant to $\bfG(s)$, i.e.,
\begin{equation} \label{eq:tan_int}
\begin{array}{ll}
\bfsfc_i^T \bfG(\mu_i) = \bfsfc_i^T \bfGr(\mu_i),  & \quad i=1,\ldots,q, \\[2mm]
\bfG(\sigma_j) \bfsfb_j = \bfGr(\sigma_j)\bfsfb_j, & \quad j=1,\ldots,r.
 \end{array}
\end{equation}
Through out the paper, we will assume $q=r$, meaning that the same number of left and right 
interpolation points are used. In addition to interpolating $\bfG(s)$, one might ask for 
matching the higher-order derivatives of $\bfG(s)$ along the tangential directions as well. 
This scenario will also be handled. 

By combining the projection-based reduced-order modeling technique with the interpolation framework, 
we want to find the $n\times r$ matrices $\bfW$ and $\bfV$ such that the reduced-order 
model (\ref{dae_rom}), (\ref{red_projection}) satisfies the tangential interpolation conditions 
(\ref{eq:tan_int}). This approach is called projection-based interpolatory model reduction. 
How to enforce the interpolation conditions via projection is shown in the following theorem,
where the $\ell$-th derivative of $\bfG(s)$ with respect to $s$
evaluated at $s=\sigma$ is denoted by $\bfG^{(\ell)}(\sigma)$. 

\begin{theorem}\textup{\cite{antoulas2010imr,gallivan2005model}}
\label{thm:interpolation_highorder}
 Let $\sigma,\,\mu \in \IC$ be such that  $s\,\bfE -\bfA$ and $s\,\bfEr -\bfAr$ are 
 both invertible for $s= \sigma,\,\mu$, and let $\bfsfb\in \IC^{m}$ and $\bfsfc\in \IC^p$ 
be fixed nontrivial vectors.
\begin{enumerate}
\item If
  \begin{equation}\label{eq:condV}
	\bigl(\left(\sigma\,\bfE-\bfA\right)^{-1}\bfE \bigr)^{j-1}\left(\sigma\,\bfE-\bfA\right)^{-1} 
         \bfB\bfsfb\in  \mbox{\normalfont\textrm{Im}}(\bfV),\enskip j=1,\ldots,N,
  \end{equation}
 then $\bfG^{(\ell)}(\sigma)\bfsfb=\bfGr^{(\ell)}(\sigma)\bfsfb$ for $\ell = 0,1,\ldots,N-1$.

\item If
  \begin{equation}\label{eq:condW}
	\bigl(\left(\mu\,\bfE-\bfA\right)^{-T}\bfE^T \bigr)^{j-1}\left(\mu\,\bfE-\bfA\right)^{-T} 
        \bfC^T\bfsfc\in  \mbox{\normalfont\textrm{Im}}(\bfW), \enskip j=1,\ldots,M,
\end{equation}
then $\bfsfc^T\bfG^{(\ell)}(\mu)=\bfsfc^T\bfGr^{(\ell)}(\mu)$ for $\ell = 0,1,\ldots,M-1$.

\item If both \textup{(\ref{eq:condV})} and \textup{(\ref{eq:condW})} hold, and if $\sigma=\mu$, then 
$\bfsfc^T\bfG^{(\ell)}(\sigma) \bfsfb =\bfsfc^T\bfGr^{(\ell)}(\sigma) \bfsfb$
for $\ell = 0,1,\ldots,M+N+1$.
\end{enumerate}
\end{theorem}

One can see that to solve the rational tangential interpolation problem via projection 
all one has to do is to construct the matrices $\bfV$ and $\bfW$ as in 
Theorem~\ref{thm:interpolation_highorder}. The dominant cost is to solve sparse linear systems. 
We also note that in Theorem \ref{thm:interpolation_highorder} the values that are interpolated 
are never explicitly computed. This is crucial since that computation is known to be  poorly conditioned \cite{feldmann1995efficient}. 

To illustrate the result of Theorem~\ref{thm:interpolation_highorder} for a~special case 
of Hermite bi-tangential interpolation, we take the same right and left interpolation points 
$\{\sigma_i\}_{i=1}^r$, left tangential directions $\{ \bfsfc_i\}_{i=1}^r$, and 
right tangential directions $\{ \bfsfb_i\}_{i=1}^r$. 
Then for the projection matrices 
\begin{eqnarray}
\bfV & = &
\left[(\sigma_1\bfE-\bfA)^{-1}\bfB\bfsfb_1,\;\,~\cdots,~\;\,
(\sigma_r\bfE-\bfA)^{-1}\bfB\bfsfb_r\,\right], \label{eqn:Vr} \\
\bfW & = & \left[ (\sigma_1\bfE-\bfA)^{-T}\bfC^T\bfsfc_1,~\cdots,~ 
(\sigma_r\bfE-\bfA)^{-T}\bfC^T \bfsfc_r \right], \label{eqn:Wr} 
\end{eqnarray}
the reduced-order model
$\bfGr(s)= \bfCr (s\bfEr - \bfAr)^{-1}\bfBr+\bfDr$ as in (\ref{red_projection})
satisfies
\begin{equation} \label{eq:tangent}
\bfG(\sigma_i) \bfsfb_i =  \bfGr(\sigma_i) \bfsfb_i,~~~
\bfsfc_i^T \bfG(\sigma_i)  =  \bfsfc_i^T \bfGr(\sigma_i),~~~
\bfsfc_i^T\bfG'(\sigma_i) \bfsfb_i =  \bfsfc_i^T\bfGr'(\sigma_i) \bfsfb_i 
\end{equation}
for $i=1,\cdots,r$, provided $\sigma_i\bfE-\bfA$ and $\sigma_i\bfEr-\bfAr$ are both nonsingular. 

Note that Theorem~\ref{thm:interpolation_highorder} does not distinguish between the singular 
$\bfE$ case
and the standard state space case with $\bfE=\bfI$. In other words, 
the interpolation conditions hold regardless as long as the matrices 
$\sigma_i \bfE - \bfA$ and $\sigma_i \bfEr - \bfAr$ are invertible. This is 
the precise reason why it is often assumed that extending interpolatory-based model 
reduction from  $\bfG(s) = \bfC(s\bfI-\bfA)^{-1}\bfB + \bfD$ to 
$\bfG(s) = \bfC(s\bfE-\bfA)^{-1}\bfB + \bfD$ is as simple as replacing $\bfI$ by $\bfE$. 
However, as the following example shows, this is not the case.

\smallskip
\begin{example}
\label{ex:1}
{\rm
Consider an~RLC circuit modeled by an index-$2$ SISO descriptor system (\ref{dae_fom})
of order $n=765$ (see, e.g., \cite{KunM06} for a~definition of index). 
We approximate this system with a~model (\ref{dae_rom}) of order $r=20$ using 
Hermite interpolation. 
The carefully chosen interpolation points were taken as the mirror images of the 
dominant poles of $\bfG(s)$. 
Since these interpolation points are known to be good points for model reduction 
\cite{gugercin2002projection,gugercin2003anmathcal},
one would expect the interpolant to be a~good approximation as well. 
However, the situation is indeed the opposite. Figure~\ref{fig:simple_example} shows 
the amplitude plots of the frequency responses $\bfG(\imath\omega)$ and 
$\bfGr(\imath\omega)$ (upper plot) and 
that of the error $\bfG(\imath\omega)-\bfGr(\imath\omega)$ (lower plot).
One can see that the error $\bfG(\imath \omega)-\bfGr(\imath \omega)$ 
grows unbounded as the frequency $\omega$ increases, and, 
hence, the approximation is extremely poor with unbounded $\Htwo$ and $\Hinf$ error norms even 
though it satisfies Hermite interpolation at carefully selected effective interpolation points.

\begin{figure}[thpb]
 \centering
\includegraphics[scale=0.45]{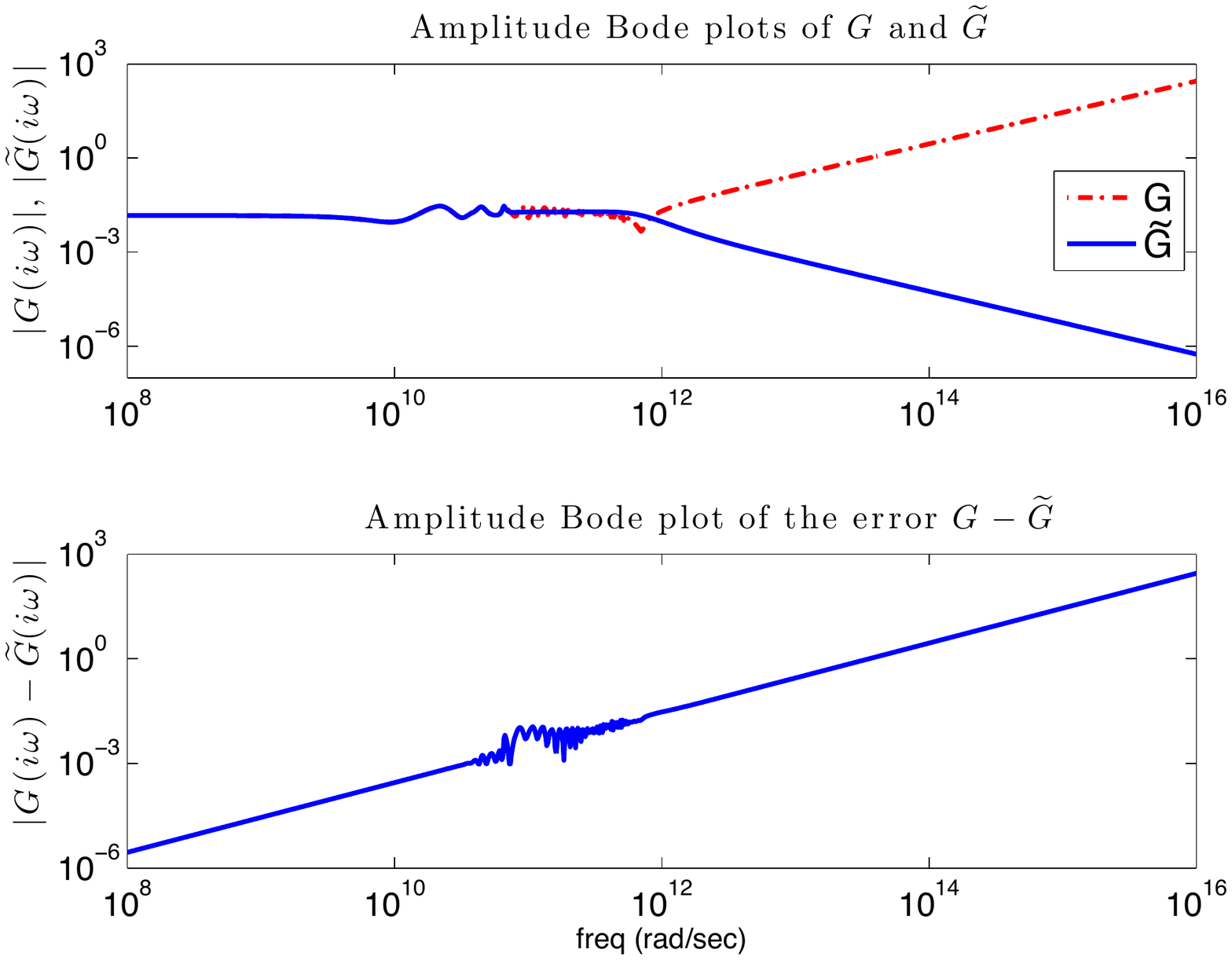} 
  \caption{\small {Example~\ref{ex:1}: amplitude plots of
$\bfG(\imath\omega)$ and $\bfGr(\imath\omega)$ (upper); the absolute error 
$|\bfG(\imath\omega)-\bfGr(\imath\omega)|$ (lower).}}
 \label{fig:simple_example}
 \end{figure}
}
\end{example}

\smallskip
The reason is simple. Even though $\bfE$ is singular, $\bfEr = \bfW ^T\bfE \bfV$
will generically be a~nonsingular matrix assuming $r \leq {\rm rank}(\bfE)$. In this case,
the transfer function $\bfGr(s)$ of the reduced-order model (\ref{dae_rom}) is proper, 
i.e., $\lim\limits_{s\to\infty}\bfGr(s)<\infty$, 
although $\bfG(s)$ might be improper.
Hence, the special care needs to be taken in order to match the polynomial part of $\bfG(s)$. 
We note that the polynomial part of $\bfGr(s)$ has to match that of $\bfG(s)$ {\it exactly}. 
Otherwise, regardless of how good the interpolation points are, the error will always 
grow unbounded. For the very special descriptor systems with the proper transfer functions and only for 
interpolation around $s=\infty$, a solution is offered in \cite{benner2006partial}.
For descriptor systems of index~1,  where the polynomial part of $\bfG(s)$ is 
a~constant matrix, a~remedy is also suggested in \cite{antoulas2010imr} by an~appropriate choice of 
$\bfDr$. However, the general case is remained unsolved. We will 
tackle precisely this problem, where (\ref{dae_fom}) is a descriptor system of higher index, 
its transfer function $\bfG(s)$ may have a~higher order polynomial part and interpolation is at arbitrary 
points in the complex plane. Thereby, the spectral projectors onto the left and right 
deflating subspaces of the pencil $\lambda \bfE-\bfA$ corresponding to finite eigenvalues 
will play a~vital role. Moreover, we will  show how to choose interpolation points and 
tangential directions optimally for interpolatory model reduction of descriptor systems.
 
\section{Interpolatory projection methods for descriptor systems} 
\label{sec:int_dae}

As stated above, in order to have bounded $\Hinf$ and $\Htwo$ errors, 
the polynomial part of $\bfGr(s)$ has to match the polynomial part of $\bfG(s)$ exactly. 
Let $\bfG(s)$ be additively decomposed as 
\begin{equation}
\bfG(s) = \bfGsp(s) + \bfGip(s),
\label{eq:GspP}
\end{equation}
where $\bfGsp(s)$ and $\bfGip(s)$ denote, respectively, the strictly proper part 
and the polynomial part of $\bfG(s)$.
We enforce the reduced-order model $\bfGr(s)$ to have the decomposition
\begin{equation} \label{eq:Gr}
\bfGr(s) = \bfGrsp(s) + \bfGrip(s)
\end{equation}
with $\bfGrip(s) =  \bfGip(s)$. This implies that the error transfer function 
does not contain a~polynomial part, i.e.,
$$
\bfG_{\rm err}(s) = \bfG(s) - \bfGr(s)  =  \bfGsp(s) - \bfGrsp(s)
$$
is strictly proper meaning $\lim\limits_{s\to\infty}\bfG_{\rm err}(s) = 0$.
Hence, by making $\bfGrsp(s)$ to interpolate $\bfGsp(s)$, we will be able 
to enforce that $\bfGr(s)$ interpolates $\bfG(s)$. This will lead to the following 
construction of $\bfGr(s)$. Given $\bfG(s)$, we create $\bfW$ and $\bfV$ satisfying new 
subspace conditions such that the reduced-order model $\bfGr(s)$ obtained   
 by projection as in (\ref{red_projection}) will not only satisfy the interpolation 
conditions but also match the polynomial part of $\bfG(s)$.
 
\begin{theorem} \label{interp_dae}
Given a~full-order model $\bfH(s) = \bfC(s\bfE-\bfA)^{-1}\bfB + \bfD$, define $\bfP_l$ 
and $\bfP_r$ to be the spectral projectors onto the left and right deflating subspaces 
of the pencil $\lambda\bfE-\bfA$ corresponding to the finite eigenvalues. Let the columns 
of $\bfW_\infty$ and $\bfV_\infty$ span the left and right deflating subspaces of 
$\lambda\bfE-\bfA$ corresponding to the eigenvalue at infinity. Let $\sigma$, 
$\mu\in \IC$ be interpolation points 
such that $s \bfE-\bfA$ and $s\bfEr-\bfAr$ are nonsingular for $s=\sigma, \mu$, 
and let $\bfsfb \in \IC^{m}$ and $\bfsfc \in \IC^{p}$. Define $\bfV_{\!f}$ and 
$\bfW_{\!f}$ such that 
\begin{align}
{\rm Im}(\bfV_{\!f}) &= \mbox{\rm span}\left\{
\left((\sigma\bfE-\bfA)^{-1}\bfE\right)^{j-1}(\sigma\bfE-\bfA)^{-1}\bfP_l\bfB \bfsfb, \enskip 
j = 1,...,N\right\}, \label{eq:vf} \\
{\rm Im}(\bfW_{\!f}) &= \mbox{\rm span}\left\{
\left((\mu\bfE-\bfA)^{-T}\bfE^T\right)^{j-1}(\mu\bfE-\bfA)^{-T}\bfP_r^T\bfC^T \bfsfc, \enskip
j = 1,...,M\right\}. \label{eq:wf}
 \end{align}
 Then with the choice of $\bfW =[\,\bfW_{\!f}, \; \bfW_\infty\,]$ and  
$\bfV=[\,\bfV_{\!f},\; \bfV_\infty\,]$, 
the reduced-order model $\widetilde{\bfG}(s) = \bfCr(s\bfEr - \bfAr)^{-1}\bfBr + \bfDr$ obtained 
via projection as in \textup{(\ref{red_projection})} satisfies
\begin{enumerate}
\item $ \widetilde{\bfP}(s) = \bfP(s)$,
\item $\bfH^{(\ell)}(\sigma)\bfsfb = \widetilde{\bfH}^{(\ell)}(\sigma)\bfsfb 
\quad \mbox{ for } \quad \ell = 0,1,\,\ldots,\,N-1,$
\item $\bfsfc^T\bfH^{(\ell)}(\mu) = \bfsfc^T\widetilde{\bfH}^{(\ell)}(\mu) 
\quad \mbox{ for } \quad \ell = 0,1,\,\ldots,\,M-1.$ 
\end{enumerate}
If $\sigma=\mu$, we have, additionally, 
$\bfsfc^T\bfH^{(\ell)}(\sigma)\bfsfb = \bfsfc^T\widetilde{\bfH}^{(\ell)}(\sigma)\bfsfb$
for $\ell = 0,\,\ldots,\,M+N+1.$
\end{theorem}

\begin{proof}
Let the pencil $\lambda \bfE -\bfA$ be transformed into the Weierstrass canonical form
\begin{equation}
\bfE= \bfS\left[\begin{array}{cc} \bfI_{n_f} & \mathbf{0} \\ \mathbf{0} & \bfN \end{array}\right]\bfT^{-1}, \qquad 
\bfA= \bfS\left[\begin{array}{cc} \bfJ & \mathbf{0} \\ \mathbf{0} & \bfI_{n_\infty} \end{array}\right]\bfT^{-1},
\label{eq:WCF}
\end{equation}
where $\bfS$ and $\bfT$ are nonsingular and $\bfN$ is nilpotent. Then the projectors
$\bfP_l$ and $\bfP_r$ can be represented as
\begin{eqnarray} \label{pil_pir}
 \bfP_l = \bfS
\left[
\begin{array}{cc}
 \bfI_{n_f} & \mathbf{0}     \\
  \mathbf{0} &   \mathbf{0} \\
\end{array}
\right]\bfS^{-1}, \qquad
\bfP_r = \bfT\left[
\begin{array}{cc}
 \bfI_{n_f} & \mathbf{0}     \\
  \mathbf{0} &   \mathbf{0} \\
\end{array}
\right]\bfT^{-1}.
\end{eqnarray}
Let $\bfT = [\,\bfT_1,\, \bfT_2\,]$ and $\bfS^{-1} = [\,\bfS_1,\, \bfS_2\,]^T$ be partitioned 
according to $\bfE$ and $\bfA$ in (\ref{eq:WCF}). Then the matrices $\bfW_\infty$ and $\bfV_\infty$
take the form
$$
\bfW_\infty=\bfS_2\bfR_S=(\bfI-\bfP_l^T)\bfW_\infty, \qquad 
\bfV_\infty=\bfT_2\bfR_T=(\bfI-\bfP_r)\bfV_\infty
$$
with nonsingular $\bfR_S$ and $\bfR_T$. Furthermore, the strictly proper and 
polynomial parts of $\bfH(s)$ in (\ref{eq:GspP}) are given by
$$
\arraycolsep=2pt
\begin{array}{rcl}
\bfH_{\rm sp}(s) & = & \bfC\bfT_1(s\bfI_{n_f}-\bfJ)^{-1}\bfS_1^T\bfB, \\[2mm]
{\rm and}~~~\bfP(s) & = & \bfC\bfT_2(s\bfN-\bfI_{n_\infty})^{-1}\bfS_2^T\bfB+\bfD,
\end{array}
$$
respectively. It follows from (\ref{eq:WCF}) and (\ref{pil_pir}) that 
$$
\begin{array}{l}
\bfE\bfP_r=\bfP_l\bfE, \qquad  \bfA\bfP_r=\bfP_l\bfA, \\
(s\bfE-\bfA)^{-1}\bfP_l=\bfP_r(s\bfE-\bfA)^{-1}, 
\end{array}
$$
and, hence, 
\begin{equation}
\bfW_{\!f}=\bfP_l^T \bfW_{\!f} \qquad{\rm and}\qquad \bfV_{\!f}=\bfP_r \bfV_{\!f}.
\label{spectral_prop2}
\end{equation}
Then the system matrices of the reduced-order model have the form
\begin{align*}
\bfEr =  \bfW^T\bfE\bfV=\left[\begin{array}{cc} 
\bfW_{\!f}^T\bfE\bfV_{\!f}^{} & \bfW_{\!f}^T\bfE\bfV_\infty^{} \\[2mm]
\bfW^T_\infty\bfE\bfV_{\!f}^{} & \bfW^T_\infty\bfE\bfV_\infty^{}\end{array}\right] =
\left[\begin{array}{cc} \bfW_{\!f}^T\bfE\bfV_{\!f}^{} & \mathbf{0} \\[2mm]
\mathbf{0} & \bfW^T_\infty\bfE\bfV_\infty^{}\end{array}\right], \\
\bfAr  =  \bfW^T\bfA\bfV=\left[\begin{array}{cc} 
\bfW_{\!f}^T\bfA\bfV_{\!f}^{} & \bfW_{\!f}^T\bfA\bfV_\infty^{} \\[2mm]
\bfW^T_\infty\bfA\bfV_{\!f}^{} & \bfW^T_\infty\bfA\bfV_\infty^{}\end{array}\right]=
\left[\begin{array}{cc} 
\bfW_{\!f}^T\bfA\bfV_{\!f}^{} & \mathbf{0} \\[2mm]
\mathbf{0} & \bfW^T_\infty\bfA\bfV_\infty^{}\end{array}\right], \\
\bfBr =  \bfW^T\bfB = \left[\begin{array}{c} 
\bfW_{\!f}^T\bfB \\[2mm] \bfW^T_\infty\bfB \end{array}\right], \qquad
\bfCr = \bfC\bfV = [\,\bfC\bfV_{\!f} ,\;\bfC\bfV_\infty \,], \qquad \bfDr=\bfD.\enskip
\end{align*}
Thus, the strictly proper and polynomial parts of $\widetilde{\bfH}(s)$ are given by
$$
\arraycolsep=2pt
\begin{array}{rcl}
\widetilde{\bfG}_{\rm sp}(s) & = & \bfC\bfV_{\!f}(s\bfW_{\!f}^T\bfE\bfV_{\!f}^{}-
\bfW_{\!f}^T\bfA\bfV_{\!f}^{})^{-1}\bfW_{\!f}^T\bfB, \\[2mm]
\widetilde{\bfP}(s) & = & \bfC\bfV_\infty(s\bfW_\infty^T\bfE\bfV_\infty^{}-
\bfW_\infty^T\bfA\bfV_\infty^{})^{-1}\bfW_\infty^T\bfB+\bfD \\[2mm]
& = & \bfC\bfT_2(s\bf I-\bfJ)^{-1}\bfS_2^T\bfB+\bfD = \bfP(s).
\end{array}
$$
One can see that the polynomial parts of $\bfH(s)$ and $\widetilde{\bfH}(s)$ coincide, 
and the proof of the interpolation result reduces to proving 
the interpolation conditions for the strictly proper parts of $\bfG(s)$ and $\bfGr(s)$.
To prove this, we first note that (\ref{eq:WCF}) and (\ref{pil_pir}) imply that
$$
\begin{array}{rcl}
 \bfC\bfP_r(\sigma\bfE-\bfA)^{-1}\bfP_l\bfB  &\!\! = \!\!& \bfC\bfT\left[
\begin{array}{cc}
 \bfI & \mathbf{0}     \\
  \mathbf{0} &   \mathbf{0} \\
\end{array}
\right]\left[\begin{array}{cc} \sigma\bfI-\bfJ & \mathbf{0} \\ 
\mathbf{0} & \sigma\bfN-\bfI \end{array}\right]^{-1}\left[
\begin{array}{cc}
 \bfI & \mathbf{0}     \\
  \mathbf{0} &   \mathbf{0} \\
\end{array}
\right]\bfS^{-1}\bfB\\[5mm]
&\!\! =\!\! & \bfC\bfT_1(\sigma\bfI-\bfJ)^{-1}\bfS_1^T\bfB = \bfG_{\rm sp}(\sigma).
\end{array}
$$
Furthermore, it follows from the relations (\ref{spectral_prop2}) that
$$
\bfC\bfP_r\bfV_{\!f} = \bfC\bfV_{\!f},  \qquad 
\bfW_{\!f}^T\bfP_l\bfB = \bfW_{\!f}^T\bfB.
$$
Due to the definitions of $\bfV_{\!f}$ and $\bfW_{\!f}$ in (\ref{eq:vf}) and 
 (\ref{eq:wf}), respectively, Theorem~\ref{thm:interpolation_highorder} gives
$$
\arraycolsep=2pt
\begin{array}{rcl}
\bfG_{\rm sp}(\sigma)\bfsfb & = & \bfC\bfP_r\bfV_{\!f}
(\sigma\bfW_{\!f}^T\bfE\bfV_{\!f} - \bfW_{\!f}^T\bfA\bfV_{\!f}^{})^{-1}\bfW_{\!f}^T\bfP_l\bfB\bfsfb=
\widetilde{\bfG}_{\rm sp}(\sigma)\bfsfb,\\
\bfsfc^T\bfG_{\rm sp}(\mu) & = & \bfsfc^T\bfC\bfP_r\bfV_{\!f}
(\mu\bfW_{\!f}^T\bfE\bfV_{\!f} - \bfW_{\!f}^T\bfA\bfV_{\!f}^{})^{-1}\bfW_{\!f}^T\bfP_l\bfB = \bfsfc^T\widetilde{\bfG}_{\rm sp}(\mu).
\end{array}
$$ 
Since both parts $1$ and $2$ of Theorem \ref{thm:interpolation_highorder} hold, we have  
$\bfsfc^T\widetilde{\bfG}_{\rm sp}'(\sigma)\bfsfb = \bfsfc^T\bfG_{\rm sp}'(\sigma)\bfsfb$ for 
 $\sigma=\mu$. The other interpolatory relations for the derivatives of the transfer function
can be proved analogously.
\end{proof} 

\medskip
Next, we illustrate that even though Theorem~\ref{interp_dae} 
has a~very similar structure to that of Theorem~\ref{thm:interpolation_highorder}, 
the saddle difference between these two results makes a~big dif\-fe\-rence in the 
resulting reduced-order model. Towards this goal, we revisit Example~\ref{ex:1}. 
We reduce the same full-order model using the same interpolation points,
but imposing the subspace conditions of 
Theorem~\ref{interp_dae}, instead.  Figure \ref{fig:simple_example_revised} 
depicts the resulting amplitude plots of $\bfG(\imath\omega)$ and $\bfGr(\imath\omega)$ 
(upper plot) and that of the error $\bfG(\imath\omega)-\bfGr(\imath\omega)$ (lower plot) 
when the new subspace conditions of Theorem~\ref{interp_dae} are used. 
Unlike the case in Example~\ref{ex:1}, where the error $\bfG(\imath \omega)-\bfGr(\imath \omega)$ 
grew unbounded, for the new reduced-order model, the maximum error is below $10^{-2}$ and the error decays 
to zero as $\omega$ approaches $\infty$, since the polynomial part is captured exactly.
  \begin{figure}[thpb]
 \centering
\includegraphics[scale=0.5]{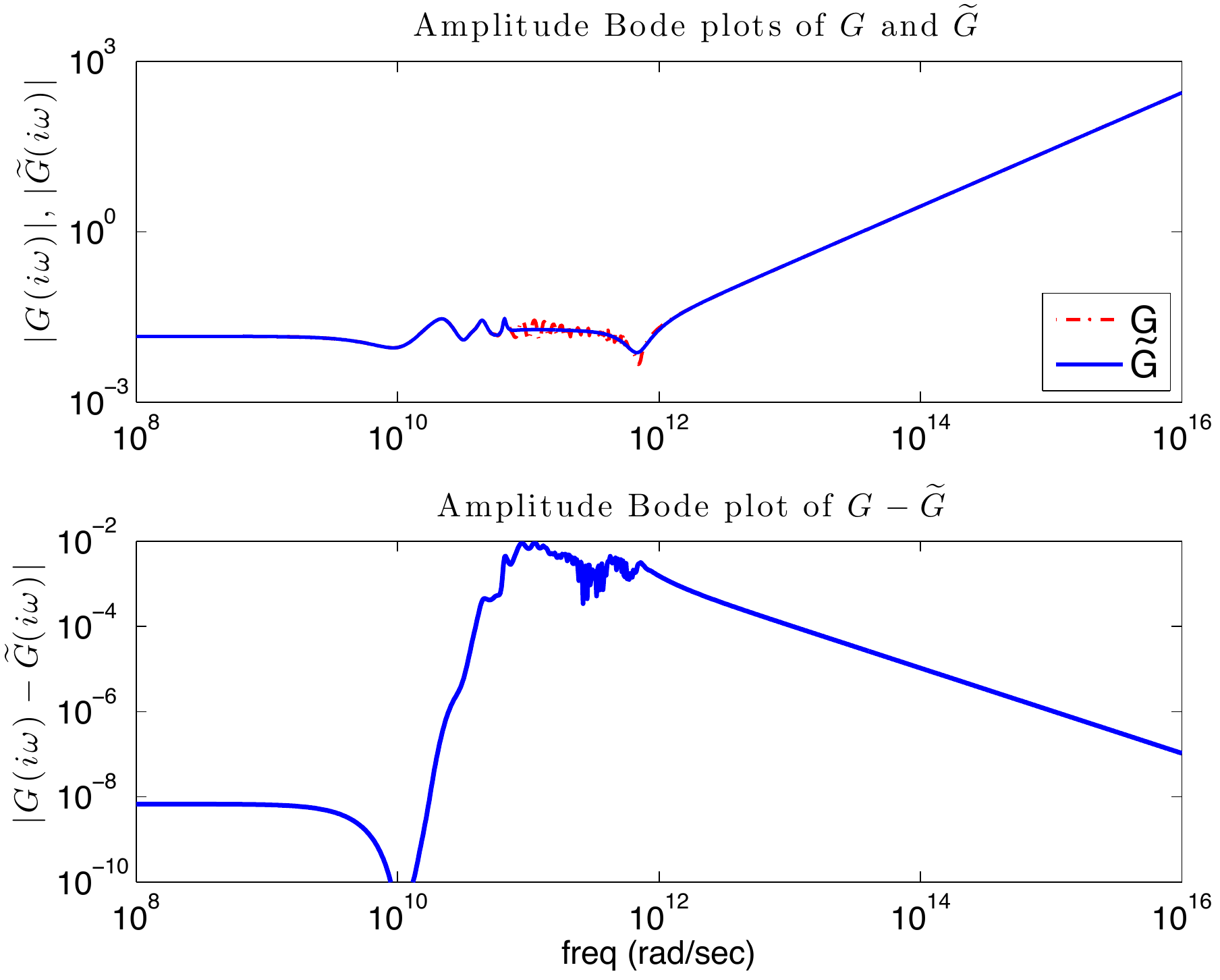} 
  \caption{\small {Amplitude plots of
$\bfG(\imath\omega)$ and $\bfGr(\imath\omega)$ (upper); 
the absolute error $|\bfG(\imath\omega)-\bfGr(\imath\omega)|$ (lower).}}
 \label{fig:simple_example_revised}
 \end{figure}
 
\smallskip
In some applications, the deflating subspaces of $\lambda \bfE-\bfA$ corresponding to the eigenvalues 
at infinity may have large dimension $n_\infty$. However, the order of the system can still
be reduced if it contains states that are uncontrollable and unobservable at infinity. 
Such states can be removed from the system without changing its transfer function 
and, hence, preserving the interpolation conditions as in Theorem~\ref{interp_dae}.
In this case the projection matrices $\bfW_\infty$ and $\bfV_\infty$ can be determined
as proposed in \cite{stykel2004gramian} by solving the projected discrete-time Lyapunov equations
\begin{eqnarray}
& & \bfA\bfX\bfA^T-\bfE\bfX\bfE^T\!=-(\bfI-\bfP_l)\bfB\bfB^T\!(\bfI-\bfP_l)^T, \;
\,\bfX=(\bfI-\bfP_r)\bfX(\bfI-\bfP_r)^T, \enskip\label{eq:GCALEc} \\
& & \bfA^T\bfY\bfA-\bfE^T\bfY\bfE=-(\bfI-\bfP_r)^T\bfC^T\!\bfC(\bfI-\bfP_r), \;
\bfY=(\bfI-\bfP_l)^T\bfY(\bfI-\bfP_l).\enskip \label{eq:GCALEo}
\end{eqnarray}
Let $\bfX_C$ and $\bfY_C$ be the Cholesky factors of $\bfX=\bfX_C^{}\bfX_C^T$ and 
$\bfY=\bfY_C^{}\bfY_C^T$, respectively, and let
$\bfY_C^T\bfA\bfX_C^{}=[\bfU_1,\,\bfU_0]\mbox{\rm diag}(\mathbf{\Sigma}, \mathbf{0})[\bfV_1,\,\bfV_0]^T$
be singular value decomposition, where $[\bfU_1,\,\bfU_0]$ and $[\bfV_1,\,\bfV_0]$ are orthogonal 
and $\mathbf{\Sigma}$ is nonsingular. Then the projection matrices 
$\bfW_\infty$ and $\bfV_\infty$ can be taken as $\bfW_\infty=\bfY_C\bfU_1$ and 
$\bfV_\infty=\bfX_C\bfV_1$. Note that the Cholesky factors $\bfX_C$ and $\bfY_C$ can be computed directly using the generalized Smith method \cite{Styk08}. In this method, it is required to solve $\nu-1$ linear systems only, where
$\nu$ is the index of the pencil $\lambda \bfE-\bfA$ or, equivalently, the nilpotence index of $\bfN$
in (\ref{eq:WCF}). The computation of the projectors $\bfP_l$ and $\bfP_r$ is, in general, 
a~difficult problem. However, for some structured problems arising in circuit simulation, multibody systems
and computational fluid dynamics, these projectors can be constructed in explicit form that significantly reduces the computational complexity of the method; see \cite{Styk08} for details. 

\section{Interpolatory optimal $\Htwo$ model reduction for descriptor systems}
\label{sec:optinterp}

The choice of interpolation points and tangential directions is the central issue
in interpolatory model reduction. This choice determines whether the reduced-order model 
is high fidelity or not. Until recently, selection of interpolation points  was largely 
\emph{ad hoc} and required several model reduction attempts to arrive at a~reasonable 
approximation. However, Gugercin {\it et al.} \cite{gugercin2008hmr} introduced 
an~interpolatory model reduction method for generating a~reduced model $\bfGr$ of 
order $r$ which is an~optimal ${\mathcal H}_2$ appro\-xi\-ma\-tion
to the original system $\bfG$ in the sense that it minimizes $\Htwo$-norm error, i.e.,
\begin{equation} \label{eq:h2prob}
 \displaystyle  \| \bfG - \bfGr \|_{\Htwo} = \min_{ {\small
\dim(\bfGr_r)=r}} \| \bfG - \bfGr_{r} \|_{\Htwo},
\end{equation}
where
\begin{equation} \label{eq:h2norm}
\left\| \bfG \right\|_{\Htwo} := \left(\frac{1}{2\pi}\int_{-\infty}^{+\infty}
\| \bfG(\imath \omega) \|_{\rm F}^2\, d\omega\right)^{1/2}
\end{equation}
and $\|\cdot \|_{\rm F}$ denotes the Frobenius matrix norm.
Since this is a~non-convex optimization problem, the computation of a~global minimizer 
is a~very difficult task. Hence, instead, one tries to find high-fidelity reduced models 
that satisfy first-order necessary optimality conditions. There exist, in general, two approaches 
for solving this problem. These are Lyapunov-based optimal $\Htwo$ methods presented in
\cite{halevi1992fwm,hyland1985theoptimal,spanos1992anewalgorithm,wilson1970optimum,
yan1999anapproximate,zigic1993contragredient} 
and interpolation-based optimal $\Htwo$ methods considered in
\cite{beattie2007kbm,beattie2009trm,bunse-gerstner2009hom,gugercin2005irk,gugercin2006rki,
gugercin2008hmr,kubalinska2007h0i,meieriii1967approximation,vandooren2008hom}. 
While the Lyapunov-based approaches require solving a~series of Lyapunov equations, 
which becomes costly and sometimes intractable in large-scale settings, 
the interpolatory approaches only require solving a~series of sparse linear systems 
and have proved to be numerically very effective.
Moreover, as shown in \cite{gugercin2008hmr},
both frameworks are theoretically equivalent that further motivates the usage of 
interpolatory model reduction techniques for the optimal $\Htwo$ approximation.

For SISO systems, interpolation-based $\Htwo$ optimality conditions were originally 
developed by Meier and Luenberger \cite{meieriii1967approximation}.  Then, based on these conditions,
an effective algorithm for interpolatory optimal $\Htwo$ approximation, called the 
\textit{Iterative Rational Krylov Algorithm} (IRKA), was introduced in \cite{gugercin2005irk, gugercin2006rki}.
This algorithm has also been recently extended to MIMO systems 
using the tangential interpolation framework, see 
\mbox{\cite{bunse-gerstner2009hom,gugercin2008hmr,vandooren2008hom}}  for more details.

The model reduction methods mentioned above, however, only deals with the system 
\mbox{$\bfG(s) = \bfC(s\bfE - \bfA)^{-1}\bfB + \bfD$} 
with a~nonsingular matrix $\bfE$. In this section, 
we will extend IRKA to descriptor systems. 
First, we establish the interpolatory $\Htwo$ optimality conditions in the new setting.

\begin{theorem}   \label{thm:h2cond}
Let $\bfG(s) = \bfGsp(s) + \bfGip(s)$ be decomposed into the strictly proper and polynomial parts, 
and let $\bfGr(s) =  \bfGrsp(s) + \bfGrip(s)$
have an $r^{\rm th}$-order strictly proper part $\bfGrsp(s)=\bfCr_{\rm sp}(s \bfEr_{\rm sp} - \bfAr_{\rm sp})^{-1}\bfBr_{\rm sp}$.
\begin{enumerate}
\item If $\bfGr(s)$ minimizes the $\Htwo$-error  $\|\bfG-\bfGr\|_{\mathcal{H}_{2}}$
over all reduced-order models with an $r^{\rm th}$-order strictly proper part, then
$\bfGrip(s) = \bfGip(s)$ and $\bfGrsp(s)$ minimizes the $\Htwo$-error $\|\bfGsp-\bfGrsp\|_{\mathcal{H}_{2}}$.
\item Suppose that the pencil $s\bfErsp-\bfArsp$  has distinct eigenvalues
$\{\widetilde{\lambda}_i\}_{i=1}^r$.
Let $\bfy_i$ and $\bfz_i$ denote the left and right eigenvectors associated with 
$\tilde{\lambda}_i$ so that \mbox{$\bfArsp \bfz_i =\!\widetilde{\lambda}_i\bfErsp\bfz_i$}, 
$\bfy_i^*\bfArsp = \!\widetilde{\lambda}_i\bfy_i^*\bfErsp$, and 
$\bfy_i^*\bfErsp \bfz_j = \delta_{ij}$. 
Then for \mbox{$\bfsfc_i= \bfCrsp \bfz_i$} and $\bfsfb_i^T = \bfy_i^*\bfBrsp$, we have
\begin{equation}\label{eq:rom}
\begin{array}{c}
\bfG(-\widetilde{\lambda}_i) \bfsfb_i =
\bfGr(-\widetilde{\lambda}_i) \bfsfb_i, \qquad 
\bfsfc_i^T\bfG(-\widetilde{\lambda}_i) =
\bfsfc_i^T \bfGr (-\widetilde{\lambda}_i), \\[1mm]
\mbox{and} \quad \bfsfc_i ^T\bfG'(-\widetilde{\lambda}_i) \bfsfb_i =
\bfsfc_i ^T \bfGr'(-\widetilde{\lambda}_i) \bfsfb_i\quad\mbox{for}\enskip i=1,\cdots,r. 
\end{array}
\end{equation}
\end{enumerate}
\end{theorem}

\begin{proof}
1. The polynomial part of $\bfG(s)$ and $\bfGr(s)$ coincide, since, otherwise, 
the $\Htwo$-norm of the error $\bfG(s)-\bfGr(s)$ would be unbounded. Then it readily follows that 
$\bfGrsp(s)$ minimizes $\|\bfGsp(s)-\bfGrsp(s)\|_{\mathcal{H}_{2}}$ since
$\bfG(s) - \bfGr(s) = \bfGsp(s) - \bfGrsp(s)$. 

2. Since $\bfGrip(s) = \bfGip(s)$, the $\Htwo$ optimal model reduction 
problem for $\bfG(s)$ now reduces to the $\Htwo$ optimal problem for the strictly proper
transfer function $\bfGsp(s)$. Hence, the optimal $\Htwo$ conditions 
of \cite{gugercin2008hmr} require that $\bfGrsp(s)$ needs to be a~bi-tan\-gen\-tial Hermite 
interpolant to $\bfGsp(s)$ with $\{-\widetilde{\lambda}_i\}_{i=1}^r$ being the interpolation points, and 
$\{\bfsfc_i\}_{i=1}^r$ and $\{\bfsfb_i\}_{i=1}^r$ being 
the corresponding left and right 
tangential  directions, respectively. Thus, the interpolation conditions (\ref{eq:rom}) hold since $\bfGrip(s) = \bfGip(s)$.
\end{proof}

Unfortunately, the $\Htwo$ optimal interpolation points and associated tangent directions
are not known {\it a priori}, since they depend on the reduced-order model to be computed. 
To overcome this difficulty, an~iterative algorithm IRKA was developed \cite{gugercin2005irk, gugercin2006rki} which is based 
on successive substitution. 
In IRKA,  the interpolation points are corrected iteratively by the choosing
mirror images of poles of the current reduced-order model as the next interpolation points.
The tangential directions are corrected in a similar way; see \cite{antoulas2010imr,gugercin2008hmr}
for details.
 
The  situation in the case of descriptor systems is similar, where the optimal interpolation 
points and the corresponding tangential directions depend on the strictly proper part 
of the reduced-order model to be computed. Moreover, we need to make sure that the final 
reduced-model has the same polynomial part as the original one. Hence, we will modify IRKA 
to meet these challenges. In particular, we will correct not the poles and the tangential directions 
of the intermediate reduced-order model at the successive iteration step but that of 
the strictly proper part of the intermediate reduced-order model. As in the case of Theorem~\ref{interp_dae}, the spectral projectors 
$\bfP_l$ and $\bfP_r$ will be used to construct the required interpolatory subspaces. 
A~sketch of the resulting model reduction method is given in Algorithm~\ref{alg:irka_dae}.

\begin{figure}[ht]
\centering
    \framebox[5.0in][t]{
    \begin{minipage}[c]{4.95in}
    \begin{algorithm} \label{alg:irka_dae}
    {\bf  Interpolatory $\Htwo$ optimal model reduction method \\ \hspace*{28mm}for descriptor systems}
    
\begin{enumerate}[{\rm 1)}]
\item Make an initial selection of the interpolation points  $\{\sigma_i\}_{i=1}^r$ and the \\
tangential directions  
$\{\bfsfb_i\}_{i=1}^r$ and  $\{\bfsfc_i\}_{i=1}^r$.

\item  $\;\bfV_{\!f} =  \left[\,(\sigma_1 \bfE- \bfA)^{-1}\bfP_l\bfB\bfsfb_1,\;\ldots,\; 
(\sigma_r \bfE - \bfA)^{-1}\bfP_l\bfB\bfsfb_r\,\right]$,\\[1mm]
   $\bfW_{\!f} =  \left[\,({\sigma_1}\bfE - \bfA)^{-T}\bfP_r^T\bfC^T\bfsfc_1,\; 
\ldots, \; ({\sigma_r} \bfE - \bfA)^{-T}\bfP_r^T\bfC^T\bfsfc_r\,\right]$.       
\item while (not converged)

       \begin{enumerate}[{\rm a)}]
       \item $\bfArsp = \bfW_{\!f}^T \bfA \bfV_{\!f}$, $\bfErsp = \bfW_{\!f}^T \bfE \bfV_{\!f}$,
        $\bfBrsp = \bfW_{\!f}^T\bfB$, and $\bfCrsp = \bfC \bfV_{\!f}$. 
       
       \item Compute  $\bfArsp \bfz_i = \widetilde{\lambda}_i\bfErsp\bfz_i$ and 
       $\bfy_i^*\bfArsp = \widetilde{\lambda}_i\bfy_i^*\bfErsp$ with 
       $\bfy_i^*\bfErsp \bfz_j = \delta_{ij}$, \\
        where $\bfy_i$ and $\bfz_i$ are left and right eigenvectors associated with  $\widetilde{\lambda}_i$.
        
\item $\sigma_i \leftarrow -\widetilde{\lambda}_i$,  $\bfsfb_i^T \leftarrow \bfy_i^*\bfBrsp$ and
 $\bfsfc_i \leftarrow \bfCrsp \bfz_i$ for $i=1,\ldots,r$.
 
   \item  $\;\bfV_{\!f} =  \left[\,(\sigma_1 \bfE- \bfA)^{-1}\bfP_l\bfB\bfsfb_1,\;\ldots,\; 
(\sigma_r \bfE - \bfA)^{-1}\bfP_l\bfB\bfsfb_r\,\right]$,\\[1mm]
   $\bfW_{\!f} =  \left[\,({\sigma_1} \bfE - \bfA)^{-T}\bfP_r^T\bfC^T\bfsfc_1,\;\ldots, \; 
   ({\sigma_r}\bfE - \bfA)^{-T}\bfP_r^T\bfC^T\bfsfc_r\,\right]$.
          \end{enumerate}
          end while
\item Compute $\bfW_\infty$ and $\bfV_\infty$ such that
$\mbox{\rm Im}(\bfW_\infty)=\mbox{\rm Im}(\bfI-\bfP_l^T)$ and \\
$\mbox{\rm Im}(\bfV_\infty)=\mbox{\rm Im}(\bfI-\bfP_r)$.

\item Set $\bfV = [\,\bfV_{\!f},\;\bfV_\infty\,]$ and $\bfW = [\,\bfW_{\!f},\;\bfW_\infty\,]$.

\item $\bfEr = \bfW^T \bfE \bfV$, $\bfAr = \bfW^T \bfA \bfV$, $\bfBr = \bfW^T \bfB$, 
$\bfCr = \bfC\bfV$,
$\bfDr = \bfD$.
\end{enumerate}
\end{algorithm}
    \end{minipage}
}
  \end{figure}
   
Note that until Step~4 of Algorithm~\ref{alg:irka_dae}, the polynomial part is not included since 
the interpolation parameters result from the strictly proper part $\bfGrsp(s)$. In a~sense,
Step~3 runs the optimal $\Htwo$ iteration on $\bfGsp(s)$. 
Hence, at the end of Step~3, we construct an optimal $\Htwo$ interpolant to $\bfGsp(s)$.
However, in Step~5, we append the interpolatory subspaces with $\bfV_\infty$ 
and $\bfW_\infty$ (which can be computed as described at the end of Section~\ref{sec:int_dae}) so that the final reduced-order model in Step~6 has the same polynomial
part as $\bfG(s)$, and, consequently, the final reduced-order model $\bfGr(s)$ satisfies the
optimality conditions of Theorem~\ref{thm:h2cond}.  One can see 
this from Step~3c: upon convergence, the interpolation points are the mirror images of 
the poles of $\bfGrsp(s)$ and the tangential directions are the residue directions from 
$\bfGrsp(s)$  as the optimality conditions require. Since Algorithm~\ref{alg:irka_dae} uses 
the projected quantities $\bfP_l\bfB$
and $\bfC\bfP_r$, theoretically iterating on a strictly proper dynamical system, the convergence behavior of this algorithm 
will follow the same pattern of IRKA which has been observed to converge rapidly in numerous numerical applications. 

Summarizing, we have shown so far how to reduce descriptor systems such that
the transfer function of the reduced descriptor systems is a~tangential interpolant 
to the original one and matches the polynomial part preventing unbounded $\Hinf$ and 
$\Htwo$ error norms. However, this model reduction approach involves the explicit computation 
of the spectral projectors or the corresponding deflating subspaces, which could be numerically infeasible 
for general large-scale problems. In the next two sections, 
we will show that for certain important classes of descriptor systems, 
the same can be achieved without explicitly forming the spectral projectors.

  \section{Semi-explicit descriptor systems of index~1} 
  \label{sec:index1}

We consider the following semi-explicit descriptor system
\begin{equation}   \label{smile}
  \arraycolsep=2pt
\begin{array}{rcl}
\bfE_{11} \dot{\bfx}_1(t)+\bfE_{12} \dot{\bfx}_2(t) & = & \bfA_{11} \bfx_1(t) +\bfA_{12}\bfx_2(t) +\bfB_1 \bfu(t), \\
 \mathbf{0} & = & \bfA_{21}\bfx_1(t) + \bfA_{22}\bfx_2(t) + \bfB_2\bfu(t), \\
 \bfy(t) &=& \bfC_1\bfx_1(t) +\bfC_2\bfx_2(t)+\bfD \bfu(t), \\
 \end{array}
\end{equation}
where the state is $\bfx(t) = [\, \bfx_1^T(t),\; \bfx_2^T(t)\,]^T \in \IR^n$ with $\bfx_1(t)\in \IR^{n_1}$, $\bfx_2(t) \in \IR^{n_2}$ and $n_1 + n_2 = n$, the input is $\bfu(t)\in \IR^{m}$, the output is $\bfy(t)\in \IR^p$, and 
$\bfE_{11}, \bfA_{11} \in \IR^{n_1 \times n_1}$, 
$\bfE_{12},\bfA_{12}\in \IR^{n_1 \times n_2}$,
$\bfA_{21}\in \IR^{n_2 \times n_1}$, 
$\bfA_{22}\in \IR^{n_2 \times n_2}$, 
$\bfB_{1}\in \IR^{n_1 \times m}$, 
$\bfB_{2}\in \IR^{n_2 \times m}$, 
$\bfC_{1}\in \IR^{p \times n_1}$, 
$\bfC_{2}\in \IR^{p \times n_2}$, 
$\bfD\in \IR^{p \times m}$. We assume that $\bfA_{22}$ and $\bfE_{11}-\bfE_{12}\bfA^{-1}_{22}\bfA_{21}$ are both nonsingular. In this case system
(\ref{smile}) is of index~1. We now compute the polynomial part of this system.

\begin{proposition} \label{index1_poly}
Let $\bfG(s)$ be a transfer function of the descriptor system \textup{(\ref{smile})}, where 
$\bfA_{22}$ and $\bfE_{11}-\bfE_{12}\bfA^{-1}_{22}\bfA_{21}$ are both nonsingular. 
Then the polynomial part of $\bfG(s)$ is a constant matrix given by  
$$
\bfP(s) = \bfC_1\bfM_1\bfB_2 + \bfC_2\bfM_2\bfB_2+\bfD,
$$ 
where
\begin{eqnarray} 
\bfM_1 & = & (\bfE_{11}-\bfE_{12}\bfA^{-1}_{22}\bfA_{21})^{-1}\bfE_{12}\bfA^{-1}_{22},
\label{M1_exp} \\
\bfM_2 & = &  -\bfA^{-1}_{22}\bfA_{21}(\bfE_{11}-\bfE_{12}\bfA^{-1}_{22}\bfA_{21})^{-1}\bfE_{12}\bfA^{-1}_{22}-\bfA_{22}^{-1}.  \label{M2_exp}
\end{eqnarray}
\end{proposition}
\begin{proof}
Consider 
$$
(s\bfE - \bfA)^{-1}\bfB = 
\left[ \begin{array}{cc} s\bfE_{11}-\bfA_{11}  & s\bfE_{12}-\bfA_{12}\\  -\bfA_{21} & -\bfA_{22} \end{array} \right]^{-1}\left[ \begin{array}{c} \bfB_1\\ \bfB_2 \end{array} \right] = 
\left[ \begin{array}{c} \bfF_1(s)\\ \bfF_2(s) \end{array} \right]. 
$$ 
This leads to 
\begin{eqnarray} \label{x1}
(s\bfE_{11}-\bfA_{11})\bfF_1(s) +  (s\bfE_{12}-\bfA_{12})\bfF_2(s) &=& \bfB_1, \\
 \label{x2}
-\bfA_{21}\bfF_1(s) -\bfA_{22}\bfF_2(s) &=& \bfB_2.
\end{eqnarray}
Solving (\ref{x2}) for $\bfF_2(s)$ gives 
$\bfF_2(s) = -\bfA_{22}^{-1}(\bfB_2 + \bfA_{21}\bfF_1(s))$, and, thus,
$$
\bfF_1(s) =  \left( (s\bfE_{11}-\bfA_{11})- (s\bfE_{12}-\bfA_{12})\bfA_{22}^{-1}\bfA_{21}\right)^{-1} \left(\bfB_1 + (s\bfE_{12}-\bfA_{12})\bfA_{22}^{-1}\bfB_2\right)
$$
implying that
$$
\lim_{s \rightarrow \infty} \bfF_1(s) = \left(\bfE_{11} - \bfE_{12}\bfA_{22}^{-1}\bfA_{21}\right)^{-1}\bfE_{12}\bfA_{22}^{-1}\bfB_2.
$$
Taking into account (\ref{x2}), we have 
$$
\lim_{s \rightarrow \infty} \bfF_2(s) = \left[-\bfA_{22}^{-1}\bfA_{21}\left(\bfE_{11} - \bfE_{12}\bfA_{22}^{-1}\bfA_{21}\right)^{-1}\bfE_{12}\bfA_{22}^{-1}-\bfA_{22}^{-1}\right]\bfB_2.$$
Finally, note that
$\bfP(s) = \lim\limits_{s \rightarrow \infty} \bfG(s) = \lim\limits_{s \rightarrow \infty} (\bfC_1\bfF_1(s) + \bfC_2\bfF_2(s)+\bfD)$, which leads to the desired conclusion.
\end{proof}

We are now ready to state the interpolation result for the descriptor system (\ref{smile}). 
This result was briefly hinted at in the recent survey \cite{antoulas2010imr}. 
Here, we present it with a~formal proof together with the formula developed 
for $\bfP(s)$ in Proposition~\ref{index1_poly}. 
As our main focus will be \mbox{$\Htwo$-based} model reduction, 
we will list the interpolation conditions only for the bi-tangential Hermite 
interpolation. Extension to the higher-order derivative interpolation is 
straightforward as shown in the earlier sections. 
\begin{lemma} \label{D_term}
Let $\bfG(s)$ be a~transfer function of the semi-explicit descriptor system 
\textup{(\ref{smile})}. For given $r$ distinct interpolation points 
$\{\sigma_i\}_{i=1}^r$, left tangential directions $\{ \bfsfc_i\}_{i=1}^r$ and
right tangential directions $\{ \bfsfb_i\}_{i=1}^r$, let 
$\bfV \in \IC^{n \times r}$ and $\bfW \in \IC^{n \times r}$ be given by
\begin{eqnarray} \label{eqn:V}
\bfV &=& 
[\,(\sigma_1\bfE-\bfA)^{-1}\bfB\bfsfb_1,\;\ldots,\;(\sigma_r\bfE-\bfA)^{-1}\bfB\bfsfb_r\,],\\ \label{eqn:W}
\bfW &=&
[\,(\sigma_1\bfE-\bfA)^{-T}\bfC^T\bfsfc_1,\; \ldots,\;(\sigma_r\bfE-\bfA)^{-T}\bfC^T\bfsfc_r\,].
\end{eqnarray}
Furthermore, let $\mathcal{B}$ and $\mathcal{C}$ be the matrices composed of the tangential directions as
\begin{equation}
\mathcal{B} = [\,\bfsfb_1, \,\ldots,\, \bfsfb_r \,] \qquad \text{and} \qquad
\mathcal{C} = [\,\bfsfc_1, \,\ldots,\, \bfsfc_r \,].
\end{equation}
Define the reduced-order system matrices as
\begin{equation}
\begin{array}{ll}
\bfEr= \bfW^{T} \bfE \bfV, &\quad 
\bfAr = \bfW^{T} \bfA \bfV + \mathcal{C}^T\bfDr\mathcal{B}, \quad 
\bfBr = \bfW^{T}\bfB - \mathcal{C}^T\bfDr, \\[.1in]
\bfCr = \bfC \bfV - \bfDr\mathcal{B}, & \quad
\bfDr  = \bfC_1\bfM_1\bfB_2 + \bfC_2\bfM_2\bfB_2+\bfD.
\end{array}
\label{eq:shiftedABC}
\end{equation}
Then the polynomial parts of $\bfGr(s)=\bfCr(s\bfEr-\bfAr)^{-1}\bfBr + \bfDr$ and $\bfG(s)$
 match assuming $\bfEr$ is nonsingular,  and $\bfGr(s)$ satisfies the bi-tangential Hermite interpolation conditions
\begin{eqnarray*} 
\bfG(\sigma_i) \bfsfb_i =  \bfGr(\sigma_i) \bfsfb_i,\qquad
\bfsfc_i^T \bfG(\sigma_i)  =  \bfsfc_i^T \bfGr(\sigma_i), \qquad
\bfsfc_i^T\bfG'(\sigma_i) \bfsfb_i =  \bfsfc_i^T\bfGr'(\sigma_i) \bfsfb_i
\end{eqnarray*}
for $i=1,\ldots,r$, provided $\sigma_i\bfE-\bfA$ and $\sigma_i\bfEr-\bfAr$ are both nonsingular.
\end{lemma}

\begin{proof}
Since $\bfEr$ is nonsingular, $\lim\limits_{s \rightarrow \infty}\bfGr(s) = \bfDr$. But by 
 Lemma~\ref{index1_poly}, we have $\widetilde{\bfD} = \lim\limits_{s \rightarrow \infty}\bfG(s)$ ensuring that the polynomial parts of $\bfG(s)$ and $\bfGr(s)$ coincide. The interpolation property is
 a~result of 
 \cite{beattie2008ipm,mayo2007fsg}, where it is shown
 that the appropriate shifting of the reduced-order quantities
 with a~non-zero feedthrough term as done in (\ref{eq:shiftedABC})  attains the original bi-tangential interpolation conditions hidden in $\bfV$ and $\bfW$ of
 (\ref{eqn:V}) and (\ref{eqn:W}), respectively. 
 \end{proof}

This result leads to Algorithm~\ref{DAEindex1_alg_1step}, which achieves bi-tangential Hermite interpolation of the semi-explicit descriptor system (\ref{smile}) without explicitly forming the spectral projectors.

\begin{figure}[ht]
\centering
    \framebox[5.1in][t]{
    \begin{minipage}[c]{4.95in}
     \begin{algorithm} \label{DAEindex1_alg_1step}
    {\bf  Interpolatory  model reduction for semi-explicit \\ \hspace*{28mm}
    descriptor systems of index 1}
\begin{enumerate}[{\rm 1)}]
\item Make an~initial selection of the interpolation points  $\{\sigma_i\}_{i=1}^r$ and the 
tangential directions $\{{\bfsfb}_i\}_{i=1}^r$ and  $\{{\bfsfc}_i\}_{i=1}^r$.
\item  $\;\bfV =  [\,(\sigma_1 \bfE- \bfA)^{-1}\bfB{\bfsfb}_1,\,\ldots, 
\,(\sigma_r \bfE - \bfA)^{-1}\bfB{\bfsfb}_r\,]$,\\[1mm]
     $ \bfW = [\,(\sigma_1\bfE-\bfA)^{-T}\bfC^T\bfsfc_1,\, \ldots,\,
(\sigma_r\bfE-\bfA)^{-T}\bfC^T\bfsfc_r\,]$.
\item Define $\bfDr  = \bfC_1\bfM_1\bfB_2 + \bfC_2\bfM_2\bfB_2+\bfD$, where $\bfM_1$ 
and $\bfM_2$ are defined in \textup{(\ref{M1_exp})} and \textup{(\ref{M2_exp})}, respectively.
\item Define $\mathcal{B} = \left[\, \bfsfb_1,\, \ldots,\, \bfsfb_r \,\right]$ and
$\mathcal{C} = \left[\, \bfsfc_1,\, \ldots,\, \bfsfc_r \,\right]$.
\item $\bfEr= \bfW^{T} \bfE \bfV$,~ 
$\bfAr = \bfW^{T} \bfA \bfV + \mathcal{C}^T\bfDr\mathcal{B}$,~ 
$\bfBr = \bfW^{T}\bfB - \mathcal{C}^T\bfDr$,~
$\bfCr = \bfC \bfV - \bfDr\mathcal{B}$.
\end{enumerate}
\end{algorithm}
     \end{minipage}
    }
  \end{figure}
 
We want to emphasize that the assumption in Lemma~\ref{D_term} that 
$\bfEr$ be nonsingular is not restrictive. This will be the case generically.
If $\bfW$ and $\bfV$ are full-rank $n \times r$ matrices, the rank of the
$r \times r $ matrix $\bfEr = \bfW^T \bfE \bfV$ will be generically $r$ 
as long as $rank(\bfE)>r$. The fact that $\bfEr$ will be full-rank is indeed 
the precise reason why we cannot simply apply 
Theorem~\ref{thm:interpolation_highorder} to descriptor systems.
   
\subsection{Optimal $\Htwo$ model reduction for semi-explicit descriptor systems}
Lemma~\ref{D_term} provides the theoretical basis for an IRKA-based 
iteration for $\Htwo$ model reduction of semi-explicit descriptor systems.
One naive approach would be the following:
Given system (\ref{smile}), simply apply IRKA of \cite{gugercin2008hmr} to obtain 
 an~intermediate reduced-order model 
$\widehat{\bfG}(s) = \widehat{\bfC}(s\widehat{\bfE} - \widehat{\bfA})^{-1}\widehat{\bfB}+\widehat{\bfD}$. Of course, this will be generically an ODE and will not necessarily match the behavior of $\bfG(s)$ around $s=\infty$. Thus, apply Lemma  \ref{D_term} to obtain the final reduced-model $\bfGr(s) = \bfCr (s\bfEr - \bfAr)^{-1}\bfBr + \widetilde{\bfD}$ 
 with 
  \begin{align}
  \bfEr=  \widehat{\bfE},\quad
  \bfAr =  \widehat{\bfA}+ \mathcal{C}^T\bfDr\mathcal{B}, \quad
  \bfBr = \widehat{\bfB} - \mathcal{C}^T\bfDr, \quad
  \bfCr = \widehat{\bfC} - \bfDr\mathcal{B}, \label{eq:naiveIRKA}
\end{align}
where $\bfDr$ is defined as in Lemma~\ref{D_term}.  While this shifting of the intermediate matrices by the $\bfD$-term  guarantees that the polynomial parts of $\bfG(s)$ and $\bfGr(s)$ match, the $\Htwo$ optimality conditions will {\it not} be satisfied.
The reason is as follows. Recall that the $\Htwo$ optimality requires bi-tangential Hermite interpolation at the mirror images of the reduced-order poles. The intermediate model $\widehat{\bfG}(s)$ satisfies this but since it does not match the polynomial part, the resulting $\Htwo$ error is unbounded. Then constructing $\bfGr(s)$ as in (\ref{eq:naiveIRKA}), we enforce the matching of the polynomial part but $\bfGr(s)$ still interpolates $\bfG(s)$ at the same interpolation points as $\widehat{\bfG}(s)$, i.e.,
at the mirror images of the poles of $\widehat{\bfG}(s)$. However, clearly due to (\ref{eq:naiveIRKA}), the poles of $\widehat{\bfG}(s)$ and $\bfGr(s)$ are different; thus 
$\bfGr(s)$ will no longer satisfy the optimal $\mathcal{H}_2$ necessary conditions. In order to achieve both the mirror-image interpolation conditions and the polynomial part matching,
the $\bfDr$ term modification must be included  throughout the iteration, not just at the end.
This results in Algorithm \ref{DAEindex1_alg}.

\begin{figure}[ht]
\centering
    \framebox[5.1in][t]{
    \begin{minipage}[c]{4.95in}
    \begin{algorithm} \label{DAEindex1_alg}
     {\bf IRKA for semi-explicit descriptor systems of index~1}
\begin{enumerate}[{\rm 1)}]
\item Make an initial shift selection $\{\sigma_i\}_{i=1}^r$ and initial tangential
directions $\{{\bfsfb}_i\}_{i=1}^r$ and $\{{\bfsfc}_i\}_{i=1}^r$.
\item  $\bfV_r = \left[\,(\sigma_1 \bfE - \bfA)^{-1}\bfB{\bfsfb}_1,\,\ldots,\,
(\sigma_r \bfE - \bfA)^{-1}\bfB{\bfsfb}_r\,\right]$, \\
$\bfW_r = \left[\,({\sigma_1} \bfE - \bfA)^{-T}\bfC^T{\bfsfc}_1,\,\ldots,\,
({\sigma_r} \bfE - \bfA)^{-T}\bfC^T{\bfsfc}_r\,\right]$.
\item Define $ \bfDr  = \bfC_1\bfM_1\bfB_2 + \bfC_2\bfM_2\bfB_2+\bfD$, 
where $\bfM_1$ and $\bfM_2$ are defined in \textup{(\ref{M1_exp})} and \textup{(\ref{M2_exp})}, respectively.
\item Define $\mathcal{B} = \left[\,\bfsfb_1,\, \ldots,\,\bfsfb_r\,\right]$ and
$\mathcal{C} = \left[\, \bfsfc_1,\,\ldots,\,\bfsfc_r\,\right]$.

\item while (not converged)
\begin{enumerate}[{\rm a)}]
\item $\bfEr= \bfW^{T} \bfE \bfV$, 
$\bfAr = \bfW^{T} \bfA \bfV + \mathcal{C}^T\bfDr\mathcal{B}$, 
$\bfBr = \bfW^{T}\bfB - \mathcal{C}^T\bfDr$,
$\bfCr = \bfC \bfV - \bfDr\mathcal{B}$.

\item Compute $\bfY^*\bfAr\bfZ = \mbox{\rm diag}(\lambda_1,\ldots,\lambda_r)$ 
and $\bfY^*\bfEr\bfZ = \bfI_r$, where the columns of $\bfZ=[\bfz_1,\ldots,\bfz_r]$
 and $\bfY=[\bfy_1,\ldots,\bfy_r]$ are, respectively, the right and left eigenvectors of $\lambda\bfEr-\bfAr$.
\item $\sigma_i \leftarrow -\lambda_i$,
${\bfsfb}_i^T \leftarrow \bfy_i^*\bfBr$ and 
${\bfsfc}_i \leftarrow \bfCr\bfz_i$ for $i=1,\ldots,r$.
\item  $\bfV = \left[\,(\sigma_1 \bfE - \bfA)^{-1}\bfB{\bfsfb}_1,\,\ldots,\,
(\sigma_r \bfE - \bfA)^{-1}\bfB{\bfsfb}_r\,\right]$, \\
$\bfW = \left[\,({\sigma_1} \bfE - \bfA)^{-T}\bfC^T{\bfsfc}_1,\,\ldots,\,
({\sigma_r} \bfE - \bfA)^{-T}\bfC^T{\bfsfc}_r\,\right]$.
\end{enumerate}
end while

\item $\bfEr= \bfW^{T} \bfE \bfV$, 
$\bfAr = \bfW^{T} \bfA \bfV + \mathcal{C}^T\bfDr\mathcal{B}$,
$\bfBr = \bfW^{T}\bfB - \mathcal{C}^T\bfDr$, 
$\bfCr = \bfC \bfV - \bfDr\mathcal{B}$.
\end{enumerate}
\end{algorithm}
    \end{minipage}
    }
  \end{figure}
  
The next result  is a restatement  of the above discussion. 

\begin{corollary}
Let $\bfG(s)$ be a~transfer function of the semi-explicit descriptor system 
\textup{(\ref{smile})}
 and let $\bfGr(s) = \bfCr (s\bfEr - \bfAr)^{-1}\bfBr + \widetilde{\bfD}$
be obtained by Algorithm~\textup{\ref{DAEindex1_alg}}. Then $\bfGr(s)$ satisfies 
the first-order necessary conditions of the $\mathcal{H}_2$ optimal model 
reduction problem.
\end{corollary}

\subsection{Supersonic inlet flow example} \label{sec:inlet}

Consider the Euler equations modelling the unsteady flow through a supersonic diffuser
as described in \cite{LassW05}. Linearization around a~steady-state solution and spatial discretization 
using a~finite volume method leads to a~semi-explicit descriptor system (\ref{smile}) of dimension $n=11730$. 
For simplicity, we focus on the single-input single-output subsystem dynamics corresponding to the input as  
the bleed actuation mass flow  and the output as the average Mach number. 

It is important to emphasize that applying balanced truncation to this model is far from trivial
because of difficulty of solving the Lyapunov equations. Instead, we apply  the proposed method in 
Algorithm~\ref{DAEindex1_alg} to obtain an $\Htwo$-optimal reduced-model of order $r=11$,
where the only cost are sparse linear solves and the need for computing the spectral projectors are removed. 
As pointed out in \cite{LassW05}, the frequencies of practical interest are the low frequency components.
Figure~\ref{figure:inlet}
shows the amplitude and phase plots of $\bfG(\imath\omega)$ and $\widetilde{\bfG}(\imath\omega)$
for $\omega\in[0,25]$ illustrating a~very accurate match of the original model. The resulting model reduction errors are 
$$ 
\frac{\|\bfG-\bfGr\|_{\mathcal{H}_{\infty}}}{  \|\bfG\|_{\mathcal{H}_{\infty}}}
  = 5.2252 \times 10^{-2} \qquad  {\rm and} \qquad 
 \frac{\|\bfG_{\rm sp}-\bfGr_{\rm sp}\|_{\mathcal{H}_{\infty}}}{  \|\bfG_{\rm sp}\|_{\mathcal{H}_{\infty}}}
 = 5.2251 \times 10^{-2}.
$$

 \begin{figure}[ht]
 \label{figure:inlet}
   \centering
 \includegraphics[scale=0.6]{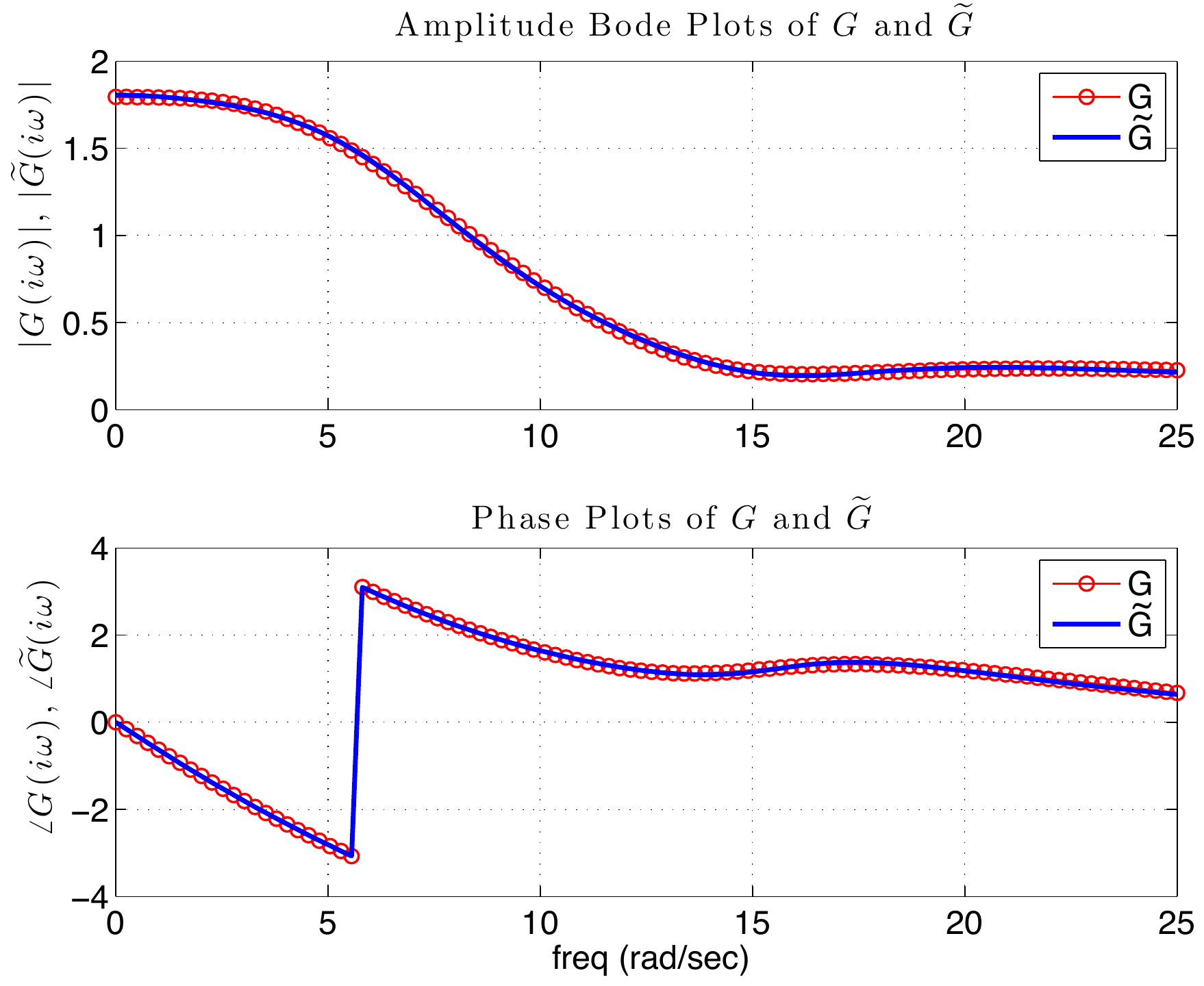}
\caption{\small {Supersonic inlet flow model: amplitude and phase Bode plots of $\bfG(s)$ and $\widetilde{\bfG}(s)$.}}
 \end{figure}
 
\section{Stokes-type descriptor systems of index~2} \label{sec:index2}

In this  section, we consider a~Stokes-type descriptor system of the form
\begin{equation}   \label{DAEsys}
\arraycolsep=2pt
\begin{array}{rcl}
 \bfE_{11} \dot{\bfx}_1(t) & = & \bfA_{11} \bfx_1(t) +\bfA_{12}\bfx_2(t) +\bfB_1 \bfu(t), \\
 \mathbf{0} & = & \bfA_{21}\bfx_1(t) + \bfB_2\bfu(t), \\
 \bfy(t) &=& \bfC_1\bfx_1(t) +\bfC_2\bfx_2(t) + \bfD\bfu(t), \\
 \end{array}
\end{equation}
where the state is $\bfx(t) = \left[ \bfx_1^T(t),\; \bfx_2^T(t) \right]^T \in \IR^n$ with 
$\bfx_1(t)\in \IR^{n_1}$, $\bfx_2(t) \in \IR^{n_2}$ and $n_1 + n_2 = n$, the input is 
$\bfu(t)\in \IR^{m}$, the output is $\bfy(t)\in\IR^p$, and 
$\bfE_{11}, \bfA_{11} \in \IR^{n_1 \times n_1}$, $\bfA_{12}\in \IR^{n_1 \times n_2}$,
$\bfA_{21}\in \IR^{n_2 \times n_1}$, $\bfB_{1}\in \IR^{n_1 \times m}$, 
$\bfB_{2}\in \IR^{n_2 \times m}$, $\bfC_{1}\in \IR^{p \times n_1}$, 
$\bfC_{2}\in \IR^{p \times n_2}$, and $\bfD \in \IR^{p \times m}$. 
We assume that $\bfE_{11}$ is nonsingular, $\bfA_{12}$ and $\bfA_{21}^T$ have both full  column rank and
$\bfA_{21}\bfE_{11}^{-1}\bfA_{12}$ is nonsingular. In this case, system (\ref{DAEsys}) is of index~2.

In \cite{heinkenschloss2008btm}, the authors showed how to apply ADI-based balanced truncation to systems of the form (\ref{DAEsys})
without explicit projector computation. 
Here, we extend this analysis to interpolatory model reduction 
and show how to reduce (\ref{DAEsys}) optimally in the $\Htwo$-norm without computing the deflating subspaces. 
Unlike \cite{heinkenschloss2008btm},  $\bfE_{11}$ is not assumed to be symmetric and positive definite,
and $\bfA_{21}$ is not assumed to be equal to $\bfA_{12}^T$. 

First, consider system (\ref{DAEsys}) with $\bfB_{2} = \mathbf{0}$, as the case of 
$\bfB_2 \neq \mathbf{0}$ follows similarly. 
Following the exposition of \cite{heinkenschloss2008btm}, consider the projectors 
$$
\arraycolsep=2pt
\begin{array}{rcl}
\bfpi_l & = & \bfI-\bfE_{11}^{-1}\bfA_{12}(\bfA_{21}\bfE_{11}^{-1}\bfA_{12})^{-1}\bfA_{21}, \\[1mm]
\bfpi_r & = & \bfI-\bfA_{12}(\bfA_{21}\bfE_{11}^{-1}\bfA_{12})^{-1}\bfA_{21}\bfE_{11}^{-1}.
\end{array}
$$
Then the descriptor system (\ref{DAEsys}) can be decoupled into a~system 
\begin{equation}   \label{37sys}
\arraycolsep=2pt
\begin{array}{rcl}
\bfpi_l\bfE_{11}\bfpi_r \dot{\bfx}_1(t) & = & \bfpi_l\bfA_{11}\bfpi_r \bfx_1(t) + \bfpi_l\bfB_1 \bfu(t) \\
 \bfy(t) &=& \bfC\bfpi_r\bfx_1(t) +\cbfD\bfu(t) 
 \end{array}
 \end{equation}
with
$$
\arraycolsep=2pt
\begin{array}{rcl}
\bfC & = & \bfC_1 - \bfC_2(\bfA_{21}\bfE_{11}^{-1}\bfA_{12})^{-1}\bfA_{21}\bfE_{11}^{-1}\bfA_{11}, \\[1mm]
\cbfD& = & \bfD - \bfC_2(\bfA_{21}\bfE_{11}^{-1}\bfA_{12})^{-1}\bfA_{21}\bfE_{11}^{-1}\bfB_1,
\end{array}
$$
and an algebraic equation 
$$
\bfx_2(t)= -(\bfA_{21}\bfE_{11}^{-1}\bfA_{12})^{-1}\bfA_{21}\bfE_{11}^{-1}\bfA_{11}\bfx_1(t)
-(\bfA_{21}\bfE_{11}^{-1}\bfA_{12})^{-1}\bfA_{21}\bfE_{11}^{-1}\bfB_1 \bfu(t).
$$
By decomposing $\bfpi_l$ and $\bfpi_r$ as
\begin{equation} \label{phi_def}
\bfpi_l = \bfTheta_{l,1}^{}\bfTheta_{l,2}^T,  \qquad \bfpi_r = \bfTheta_{r,1}^{}\bfTheta_{r,2}^T
 \end{equation}
with $\bfTheta_{l,j}, \bfTheta_{r,j} \in \IR^{n_1 \times (n_1 - n_2)}$ such that 
\begin{equation}
\bfTheta_{l,2}^T\bfTheta_{l,1}^{} = \bfI, \qquad 
\bfTheta_{r,2}^T\bfTheta_{r,1}^{} = \bfI,
\label{eq:Theta}
\end{equation} 
and defining $\tilde{\bfx}_1(t) = \bfTheta_{r,2}^T\bfx_1(t)$, system (\ref{37sys}) becomes
\begin{equation}   \label{310sys}
\arraycolsep=2pt
  \begin{array}{rcl}
\bfTheta_{l,2}^T\bfE_{11}\bfTheta_{r,1}^{} \dot{\tilde{\bfx}}_1(t) & = & 
\bfTheta_{l,2}^T\bfA_{11}\bfTheta_{r,1} \tilde{\bfx}_1(t) + \bfTheta_{l,2}^T\bfB_1 \bfu(t) \\
 \bfy(t) &=& \bfC\bfTheta_{r,1}\tilde{\bfx}_1(t) +\cbfD\bfu(t).  
 \end{array}
 \end{equation}
Then the reduction of the descriptor system (\ref{DAEsys}) is equivalent to the reduction of system (\ref{37sys}) 
or (\ref{310sys}). However, the beauty of this equivalence lies in the observation that the matrix 
$\bfTheta_{l,2}^T\bfE_{11}\bfTheta_{r,1}^{}$ is nonsingular.
Therefore, standard model reduction procedures for ODEs can be applied to system (\ref{310sys}),
and the obtained reduced-order
model will approximate the descriptor system (\ref{DAEsys}). It is important to emphasize that even though 
(\ref{37sys}) and (\ref{310sys}) are equivalent to (\ref{DAEsys}), the ultimate goal of this section is 
to develop an interpolatory model reduction method that does not require the explicit computation of either 
the projectors $\bfpi_l$, $\bfpi_r$ or the basis matrices $\bfTheta_{l,2}$, $\bfTheta_{r,1}$. 
For this purpose,  
define the matrices 
\begin{equation} 
\label{tilde_def}
\cbfE = \bfpi_l\bfE_{11}\bfpi_r, \qquad 
\cbfA = \bfpi_l\bfA_{11}\bfpi_r, \qquad 
\cbfB = \bfpi_l\bfB_1, \quad 
\cbfC= \bfC\bfpi_r.
\end{equation}
In interpolation setting, the matrix of interest will be $\sigma\cbfE - \cbfA$ with $\sigma\in \IC$. 
Luckily, several key properties of $\cbfE + \tau\cbfA$, $\tau \in \IC$, were already introduced in 
\cite{heinkenschloss2008btm}.
However, we present these results in terms of $\sigma\cbfE - \cbfA$ instead of $\cbfE + \tau\cbfA$.

\begin{lemma} \label{rest_inv_lemma}
Let $\bfTheta_{l,2}$ and $\bfTheta_{r,1}$ be the matrices defined in \textup{(\ref{phi_def})} and let 
$\sigma \in \IC$ be such that $\sigma\bfTheta_{l,2}^T\bfE_{11}\bfTheta_{r,1}^{} - 
\bfTheta_{l,2}^T\bfA_{11}\bfTheta_{r,1}^{}$ is nonsingular. The matrix defined as
\begin{align} \label{restrictedinv}
(\sigma\cbfE-\cbfA)^I :=  \bfTheta_{r,1}(\sigma\bfTheta_{l,2}^T\bfE_{11}\bfTheta_{r,1}^{} - 
\bfTheta_{l,2}^T\bfA_{11}\bfTheta_{r,1})^{-1}\bfTheta_{l,2}^T
\end{align}
satisfies
\begin{align*}
(\sigma\cbfE-\cbfA)^I(\sigma\cbfE-\cbfA) = \bfpi_r \quad \text{and} \quad
(\sigma\cbfE-\cbfA)(\sigma\cbfE-\cbfA)^I = \bfpi_l.
\end{align*}
Similarly, the matrix defined as
\begin{align} \label{restrictedinv_W}
(\sigma\cbfE^T-\cbfA^T)^I :=  \bfTheta_{l,2}(\sigma\bfTheta_{r,1}^T\bfE_{11}^T\bfTheta_{l,2} - 
\bfTheta_{r,1}^T\bfA_{11}^T\bfTheta_{l,2})^{-1}\bfTheta_{r,1}^T
\end{align}
satisfies
\begin{align*}
(\sigma\cbfE^T-\cbfA^T)^I(\sigma\cbfE^T-\cbfA^T) = \bfpi_l^T \quad \text{and} \quad
(\sigma\cbfE^T-\cbfA^T)(\sigma\cbfE^T-\cbfA^T)^I = \bfpi_r^T.
\end{align*}
\end{lemma}
\begin{proof}
Following a~similar argument to that in \cite{heinkenschloss2008btm}, the proof of the first equality 
follows directly from (\ref{phi_def}) and (\ref{restrictedinv}). Indeed, we have
\begin{align*}
(\sigma\cbfE-&\cbfA)^I(\sigma\cbfE-\cbfA)
= \bfTheta_{r,1}^{}(\sigma\bfTheta_{l,2}^T\bfE_{11}\bfTheta_{r,1}^{} - \bfTheta_{l,2}^T\bfA_{11}\bfTheta_{r,1}^{})^{-1}
\bfTheta_{l,2}^T\bfpi_l(\sigma\bfE_{11}-\bfA_{11})\bfpi_r \\
&= \bfTheta_{r,1}^{}(\sigma\bfTheta_{l,2}^T\bfE_{11}\bfTheta_{r,1}^{} - \bfTheta_{l,2}^T\bfA_{11}\bfTheta_{r,1}^{})^{-1}
\bfTheta_{l,2}^T(\sigma\bfE_{11}-\bfA_{11})\bfTheta_{r,1}^{}\bfTheta_{r,2}^T \\
&= \bfTheta_{r,1}^{}\bfTheta_{r,2}^T = \bfpi_r.
\end{align*}
The remaining equalities follow similarly.
 \end{proof}
 
 \smallskip
At first glance, the definition of the generalized inverses in (\ref{restrictedinv}) and (\ref{restrictedinv_W}) 
may seem to be irrelevant for model reduction of the descriptor system (\ref{DAEsys}). 
Recall that reducing (\ref{DAEsys}) is equivalent to reducing system (\ref{37sys}) and 
the interpolatory projection method for (\ref{37sys}) will require 
inverting $(\sigma\cbfE-\cbfA)$ and $(\sigma\cbfE^T-\cbfA^T)$. However, these inverses do not exist. 
As a~result, definitions (\ref{restrictedinv}) and (\ref{restrictedinv_W}) become pivotal in order to achieve 
interpolatory model reduction of (\ref{37sys}) and, thereby,  of (\ref{DAEsys}) as shown in the next theorem.

\begin{theorem} \label{hess_interp}
Let $s = \sigma,\mu \in \IC$ be such that the matrices
$$
s\bfTheta_{l,2}^T\bfE_{11}\bfTheta_{r,1}^{} - \bfTheta_{l,2}^T\bfA_{11}\bfTheta_{r,1}^{} 
\qquad \mbox{and}\qquad 
s\bfW^T\bfE_{11}\bfV - \bfW^T\bfA_{11}\bfV
$$ 
are invertible. Define the reduced-order model
\begin{align} \label{eq:ind2rom}
\bfGr(s) &= \bfC\bfV(s\bfW^T\bfE_{11}\bfV - \bfW^T\bfA_{11}\bfV)^{-1}\bfW^T\bfB_1 + \cbfD.
\end{align}
Let $\bfsfb \in \IC^{m}$ and $\bfsfc \in \IC^p$ be fixed nontrivial vectors. 
\begin{enumerate}[{\rm 1.}]
\item If $(\sigma\cbfE - \cbfA)^I\cbfB\bfsfb \in \mathrm{Im}(\bfV)\subset \mathrm{Im}(\bfTheta_{r,1})$ and 
$(\mu\cbfE^T - \cbfA^T)^I\cbfC^T\bfsfc \in \mathrm{Im}(\bfW)\subset \mathrm{Im}(\bfTheta_{l,2})$, 
then $ {\bfH}(\sigma)\bfsfb = \bfGr(\sigma)\bfsfb$
 and  $ \bfsfc^T{\bfH}(\mu) = \bfsfc^T\bfGr(\mu).$
\item If, in addition, $\sigma = \mu$, then $ \bfsfc^T\bfG'(\sigma)\bfsfb = \bfsfc^T\bfGr'(\sigma)\bfsfb.$
\end{enumerate}
\end{theorem}

\bigskip
\begin{remark}
Before presenting the proof, we want to emphasize that this interpolation result is different than
the usual interpolation framework given in Theorem~\textup{\ref{thm:interpolation_highorder}}, 
where the projection matrices $\bfV$ and $\bfW$ are constructed using 
$\bfA$, $\bfE$, $\bfB$ and $\bfC$ and then the projection is applied to the same quantities. 
In Theorem~\textup{\ref{hess_interp}}, however, the projection matrices $\bfV$ and $\bfW$ are constructed 
using the system matrices of \textup{(\ref{37sys})}, namely $\cbfA, \cbfE, \cbfB$ and $\cbfC$. 
But then the projection (model reduction) is applied to the system matrices of \textup{(\ref{DAEsys})}, 
namely $\bfE_{11}, \bfA_{11}, \bfB_1$ and $\bfC$. Thus, the proof will serve  to fill in this important gap. 
\end{remark}

\bigskip
\begin{proof} Since systems (\ref{DAEsys}) and (\ref{310sys}) are equivalent, they have the same transfer function given by
$$
{\bfH}(s) = \bfC\bfTheta_{r,1}^{}(s\bfTheta_{l,2}^T\bfE_{11}\bfTheta_{r,1}^{} - \bfTheta_{l,2}^T\bfA_{11}\bfTheta_{r,1}^{})^{-1}
\bfTheta_{l,2}^T\bfB_1 + \cbfD. 
$$
Since $\bfTheta_{l,2}^T\bfE_{11}\bfTheta_{r,1}^{}$ in (\ref{310sys}) is nonsingular, we make use of Theorem~\ref{thm:interpolation_highorder}.
Define $\widetilde{\bfV}$ and $\widetilde{\bfW}$ such that 
 \begin{eqnarray} \label{v_def}
 \bfV = \bfTheta_{r,1}\widetilde\bfV \qquad \text{and} \qquad
 \bfW = \bfTheta_{l,2}\widetilde\bfW.
 \end{eqnarray}
 Pluging these matrices into (\ref{eq:ind2rom}), we obtain that
\begin{align*}
\widetilde{\bfG}(s) = \bfC\bfTheta_{r,1}\widetilde\bfV(s\widetilde\bfW^T\bfTheta_{l,2}^T
\bfE_{11}\bfTheta_{r,1}\widetilde\bfV - \widetilde\bfW^T\bfTheta_{l,2}^T\bfA_{11}\bfTheta_{r,1}\widetilde\bfV)^{-1}
\widetilde\bfW^T\bfTheta_{l,2}^T\bfB_1 + \cbfD.
\end{align*}
Moreover, it follows from (\ref{eq:Theta}) that $\widetilde\bfV = \bfTheta_{r,2}^T \bfV$ and 
 $\widetilde\bfW = \bfTheta_{l,1}^T \bfW$.
To prove the first claim in part $1$, we note that (\ref{phi_def}) implies that
\begin{eqnarray} \label{Btilde}
\bfTheta_{l,2}^T\cbfB = \bfTheta_{l,2}^T\bfpi_l\bfB_1 = \bfTheta_{l,2}^T\bfTheta_{l,1}^{}\bfTheta_{l,2}^T\bfB_1 = 
\bfTheta_{l,2}^T\bfB_1.
\end{eqnarray}
Since $(\sigma\cbfE - \cbfA)^I\cbfB\bfsfb \in \mathrm{Im}(\bfV)$, there exists $\bfq \in \IR^r$ such that
$(\sigma\cbfE - \cbfA)^I\cbfB\bfsfb =\bfV\bfq$.
Using (\ref{restrictedinv}) (\ref{v_def}) and  (\ref{Btilde}), this equation can be written as 
\begin{align*}
\bfTheta_{r,1}^{}(\sigma\bfTheta_{l,2}^T\bfE_{11}\bfTheta_{r,1}^{} - \bfTheta_{l,2}^T\bfA_{11}\bfTheta_{r,1}^{})^{-1}
\bfTheta_{l,2}^T\bfB_1\bfb = \bfTheta_{r,1}\widetilde\bfV\bfq.
\end{align*}
The left multiplication by $\bfTheta_{r,2}^T$ gives
\begin{align*}
(\sigma\bfTheta_{l,2}^T\bfE_{11}\bfTheta_{r,1}^{} - \bfTheta_{l,2}^T\bfA_{11}\bfTheta_{r,1}^{})^{-1}
\bfTheta_{l,2}^T\bfB_1\bfsfb  = \widetilde\bfV\bfq.
\end{align*}
Hence, $(\sigma\bfTheta_{l,2}^T\bfE_{11}\bfTheta_{r,1}^{} - \bfTheta_{l,2}^T\bfA_{11}\bfTheta_{r,1}^{})^{-1}
\bfTheta_{l,2}^T\bfB_1\bfsfb \in \mathrm{Im}(\widetilde\bfV)$. Then it follows from 
Theorem~\ref{thm:interpolation_highorder} that  ${\bfH}(\sigma)\bfsfb = \widetilde{\bfH}(\sigma)\bfsfb$. 
The equation $\bfsfc^T \bfG(\sigma) = \bfsfc^T\widetilde{\bfH}(\sigma)$ can be obtained similarly. 
The proof of part $2$ follows from part $3$ of Theorem~\ref{thm:interpolation_highorder}.  
\end{proof}

It should be noted that the conditions $\mathrm{Im}(\bfV) \subset \mathrm{Im}(\bfTheta_{r,1})$
and $\mathrm{Im}(\bfW) \subset \mathrm{Im}(\bfTheta_{l,2})$ in part 1 of Theorem~\ref{hess_interp} 
are automatically fulfilled if for given interpolation points $\{\sigma_i\}_{i=1}^r$, $\{\mu_i\}_{i=1}^r$ 
and tangential directions $\{\bfsfb_i\}_{i=1}^r$, $\{\bfsfc_i\}_{i=1}^r$, we choose
$$
\arraycolsep =2pt
\begin{array}{rcl}
\mathrm{Im}(\bfV) & = & \mathrm{span}\{(\sigma_1\cbfE - \cbfA)^I\cbfB\bfsfb_1,\ldots,(\sigma_r\cbfE - \cbfA)^I\cbfB\bfsfb_r\}, \\
\mathrm{Im}(\bfW) & = & \mathrm{span}\{(\mu_1\cbfE^T - \cbfA^T)^I\cbfC^T\bfsfc_1,\ldots,(\sigma_r\cbfE^T - \cbfA^T)^I\cbfC^T
\bfsfc_r\}.
\end{array}
$$
 
 \subsection{Computational issues related to the reduction of index-2 descriptor systems} 
  \label{sec:saddle}

Even though Theorem~\ref{hess_interp} shows how to enforce interpolation for the descriptor system 
(\ref{DAEsys}),
the spectral projectors are still implicitly hidden in the de\-fi\-ni\-tions of 
$(\sigma\cbfE - \cbfA)^I$ and 
$(\sigma\cbfE^T - \cbfA^T)^I$.
It has been shown in \cite{heinkenschloss2008btm} how to compute the matrix-vector product
$(\cbfE + \tau\cbfA)^I \bff$ for a~given vector $\bff$ 
without explicitly forming $(\cbfE + \tau\cbfA)^I$. This approach can also be used in interpolatory model reduction, 
where the quantities of interest are  
$(\sigma\cbfE - \cbfA)^I\cbfB\bfsfb$ and $(\mu\cbfE^T - \cbfA^T)^I\cbfC^T\bfsfc$.
The proof of the following result is analogous to those in \cite{heinkenschloss2008btm}, and, therefore, it is omitted.

 \begin{lemma} \label{compute_v}
 Let $s=\sigma,\mu$ be such that 
$s\bfTheta_{l,2}^T\bfE_{11}\bfTheta_{r,1}^{} - \bfTheta_{l,2}^T\bfA_{11}\bfTheta_{r,1}^{}$ is invertible. 
Then the vector
 \begin{eqnarray} \label{saddle_B_star}
\bfv =  (\sigma\cbfE - \cbfA)^I\cbfB\bfsfb
 \end{eqnarray}
 solves
 \begin{eqnarray} \label{saddle_B}
 \left[ \begin{array}{cc} \sigma\bfE_{11} - \bfA_{11} & \bfA_{12}\\  \bfA_{21} & \mathbf{0} \end{array} \right]
\left[ \begin{array}{cc} \bfv \\  \bfz \end{array} \right]=\left[ \begin{array}{cc} \bfB_1\bfsfb \\  
\mathbf{0} \end{array} \right],
 \end{eqnarray}
 and 
the vector
 \begin{eqnarray} \label{saddle_C_star}
\bfw =  (\mu\cbfE^T - \cbfA^T)^I\cbfC^T\bfsfc
 \end{eqnarray}
 solves
 \begin{eqnarray} \label{saddle_C}
 \left[ \begin{array}{cc} \mu\bfE_{11}^T - \bfA_{11}^T & \bfA_{21}^T\\  \bfA_{12}^T & \mathbf{0} \end{array} \right]
\left[ \begin{array}{cc} \bfw \\  \bfq \end{array} \right]=\left[ \begin{array}{cc} \bfC^T\bfsfc \\  
\mathbf{0} \end{array} \right].
 \end{eqnarray}
\end{lemma}

\medskip
From a computational perspective of implementing Theorem \ref{hess_interp}, 
the importance of this result is clear.  To achieve interpolation, Theorem \ref{hess_interp} 
relies on computing the quantities 
$(\sigma\cbfE - \cbfA)^I\cbfB\bfsfb$ and $(\sigma\cbfE^T - \cbfA^T)^I\cbfC^T\bfsfc,$ both of which involve 
the computation of $\bfTheta_{l,2}$  and $\bfTheta_{r,1}$. However, Lemma~\ref{compute_v} illustrates 
that the computation of these basis matrices is unnecessary and only the linear systems  (\ref{saddle_B})
and (\ref{saddle_C}) need to be solved. 
This observation leads to Algorithm \ref{index2_alg} below for interpolatory model reduction of Stokes-type 
descriptor systems of index~2. 

\begin{figure}[ht]
\centering
    \framebox[5.1in][t]{
    \begin{minipage}[c]{4.95in}
    \begin{algorithm} \label{index2_alg}
    {\bf  Interpolatory model reduction for Stokes-type \\ \hspace*{28mm} 
    descriptor systems of index~2}
\begin{enumerate}[{\rm 1)}]
\item Make an initial selection of the interpolation points  $\{\sigma_i\}_{i=1}^r$ and the 
tangent directions
$\{{\bfsfb}_i\}_{i=1}^r$ and  $\{{\bfsfc}_i\}_{i=1}^r$.
\item For $i = 1, \ldots, r$, solve
$$ 
\left[ \begin{array}{cc} \sigma_i\bfE_{11}-\bfA_{11} & \bfA_{12}\\  \bfA_{21} & \mathbf{0} \end{array} \right]
\left[ \begin{array}{cc} \bfv_i \\  \bfz \end{array} \right]= 
\left[ \begin{array}{cc} \bfB_1\bfsfb_i \\  \mathbf{0} \end{array} \right],
$$
$$
\left[ \begin{array}{cc} \sigma_i\bfE_{11}^T - \bfA_{11}^T& \bfA_{21}^T\\  \bfA_{12}^T & \mathbf{0} \end{array} \right]
\left[ \begin{array}{cc} \bfw_i \\  \bfq \end{array} \right]= 
\left[ \begin{array}{cc} \bfC^T\bfsfc_i \\  \mathbf{0} \end{array} \right].
$$
\item $\bfV =  \left[\bfv_1,\dots, \bfv_r\right]$, \quad
       $\bfW =  \left[\bfw_1,\dots, \bfw_r\right]$.
\item $\bfEr= \bfW^T \bfE_{11} \bfV$, $\bfAr = \bfW^T \bfA_{11} \bfV$, $\bfBr = \bfW^T \bfB_1$, $\bfCr = \bfC\bfV$,
$\bfDr = \cbfD.$
\end{enumerate}
\end{algorithm}
    \end{minipage}
    }
  \end{figure}

Once a~computationally effective bi-tangential Hermite interpolation framework is established for index-2 descriptor systems, 
extending it to optimal $\Htwo$ model reduction via IRKA is straightforward and given in Algorithm \ref{index2_IRKA}. 
It follows from the structure of this algorithm that upon convergence
  the reduced model  \mbox{$\bfGr(s)= \bfCr (s\bfEr\! -\! \bfAr)^{-1}\bfBr\! +\! \widetilde{\bfD}$}
  satisfies the first-order conditions for $\Htwo$ optimality.
\begin{figure}[ht]
\centering
    \framebox[5.1in][t]{
    \begin{minipage}[c]{4.95in}
 \begin{algorithm} \label{index2_IRKA}
     {\bf IRKA for Stokes-type descriptor system of index~2}
\begin{enumerate}[{\rm 1)}]
\item Make an initial shift selection $\{\sigma_i\}_{i=1}^r$ and initial tangent directions 
$\{{\bfsfb}_i\}_{i=1}^r$ and $\{{\bfsfc}_i\}_{i=1}^r$.
\item Apply Algorithm~\textup{\ref{index2_alg}} to obtain  $\bfEr$,  $\bfAr$, $\bfBr$, $\bfCr$ and 
$\bfDr$.
\item while (not converged)
\begin{enumerate}[{\rm a)}]
\item Compute $\bfY^*\bfAr\bfZ = diag(\lambda_1,\ldots,\lambda_r)$ and $\bfY^*\bfEr\bfZ = \bfI$, 
where the columns of $\bfZ=[\bfz_1,\ldots,\bfz_r]$ and $\bfY=[\bfy_1,\ldots,\bfy_r]$ are, respectively, 
the right and left eigenvectors of $\lambda\bfEr-\bfAr$.
\item $\sigma_i \leftarrow -\lambda_i$,
${\bfsfb}_i^T \leftarrow \bfy_i^*\bfBr$ and ${\bfsfc}_i \leftarrow \bfCr\bfz_i$  for $i=1,\ldots,r$.

\item Apply Algorithm~\textup{\ref{index2_alg}} to obtain  $\bfEr$,  $\bfAr$, $\bfBr$, $\bfCr$ and 
$\bfDr$.
\end{enumerate}
end while

\end{enumerate}
\end{algorithm}
    \end{minipage}
    }
  \end{figure}

\begin{remark} 
As shown in \textup{\cite{heinkenschloss2008btm}}, the general case $\bfB_2 \neq \mathbf{0}$ can be handled similar
to the case $\bfB_2 = \mathbf{0}$.  First note that the state $\bfx_1(t)$ can be
decomposed as \mbox{$\bfx_1(t) = \bfx_0(t) + \bfx_g(t)$}, where
$\bfx_g(t) = -\bfE_{11}^{-1}\bfA_{12}(\bfA_{21}\bfE_{11}^{-1}\bfA_{12})^{-1}\bfB_2\bfu(t)$
and $\bfx_0(t)$ sa\-tis\-fies $\bfA_{21}\bfx_0(t) = 0$. After some algebraic manipulations, this leads to
\begin{equation}   \label{37sys_B2}
\arraycolsep=2pt
\begin{array}{rcl}
\bfpi_l\bfE_{11}\bfpi_r \dot{\bfx}_0(t) & = & \bfpi_l\bfA_{11}\bfpi_r \bfx_0(t) + 
\bfpi_l\bfB\bfu(t), \\
 \bfy(t) &=& \bfC\bfpi_r\bfx_0(t) + \cbfD \bfu(t) -\bfC_2(\bfA_{21}\bfE_{11}^{-1}\bfA_{12})^{-1}\bfB_2\dot{\bfu}(t), \\
 \end{array}
 \end{equation}
 where
 \begin{eqnarray}
 \bfC = \bfC_1 - \bfC_2(\bfA_{21}\bfE_{11}^{-1}\bfA_{12})^{-1}\bfA_{21}\bfE_{11}^{-1}\bfA_{11}, \\
 \bfB = \bfB_1 - \bfA_{11}\bfE_{11}^{-1}\bfA_{12}(\bfA_{21}\bfE_{11}^{-1}\bfA_{12})^{-1}\bfB_2, \\
 \cbfD = \bfD - \bfC_2(\bfA_{21}\bfE_{11}^{-1}\bfA_{12})^{-1}\bfA_{21}\bfE_{11}^{-1}\bfB_1. \label{Dtilde}
  \end{eqnarray}
Therefore, the $\bfB_2 \neq 0$ case extends to the interpolation framework as well
by defining
   \begin{align*}
\widehat{\cbfE} = \bfpi_l\bfE_{11}\bfpi_r, \quad \widehat{\cbfA} = \bfpi_l\bfA_{11}\bfpi_r, \quad 
\widehat{\cbfB} = \bfpi_l\bfB, \quad \widehat{\cbfC} = \bfC\bfpi_r
\end{align*}
and applying Theorem~\textup{\ref{hess_interp}} with $\widehat{\cbfE}$, $\widehat{\cbfA}$, 
$\widehat{\cbfB}$, $\widehat{\cbfC}$ and 
$\widehat{\cbfD}  = \cbfD - s\bfC_2(\bfA_{21}\bfE_{11}^{-1}\bfA_{12})^{-1}\bfB_2$ instead 
of $\cbfE$, $\cbfA$, $\cbfB$, $\cbfC$ and $\cbfD$. 
\end{remark}

\subsection{Numerical results for Oseen equations} \label{sec:oseen_example}

The model borrowed from \cite{heinkenschloss2008btm} is obtained by  discretizing the Oseen equations 
and describe the flow of a~viscous and incompressible fluid in a~domain $\Omega \in \IR^2$ representing 
a~channel with a~backward facing step. 
A~spatial discretization using the finite element method leads to the index-$2$ descriptor system
(\ref{DAEsys}) with $\bfE_{11}, \bfA_{11} \in \IR^{5520 \times 5520}$, 
$\bfA_{12}^{}, \bfA_{21}^T\in \IR^{5520 \times 761}$, 
$\bfB_{1}\in \IR^{5520 \times 6}$, 
$\bfB_{2}\in \IR^{761 \times 6}$, 
$\bfC_{1}\in \IR^{2 \times 5520}$, 
$\bfC_{2}\in \IR^{2 \times 761}$, $\bfD=0$, see \cite{heinkenschloss2008btm} for more details on the model. 
Note that $\bfB_2 \neq 0$ and the transfer function grows unbounded around $s=\infty$.

We approximate this system by a~model of order $r = 20$ using the balanced truncation method as described 
in \cite{heinkenschloss2008btm} and the $\Htwo$ optimal model reduction method given in Algorithm~\ref{index2_IRKA}. 
The amplitude Bode plots of the full model and two reduced-order models depicted in Figure~\ref{fig:oseen_bode} 
clearly illustrate that  interpolation-based Algorithm~\ref{index2_IRKA}
leads to a high-fidelity reduced model replicating  the full-order transfer function with almost no loss of 
accuracy and matching the performance of the balanced truncation method. The accuracy of this 
interpolation-based method  is due to the fact that  we do not choose the interpolation points in an {\it ad hoc} fashion; 
instead Algorithm~\ref{index2_IRKA} iteratively leads to $\Htwo$ optimal interpolation points.   
As the difficulty in computing $\Htwo$ norm of the error is clear, 
we approximately compute the relative $\Hinf$-error $\frac{\|\bfH_{\rm sp}-\tilde{\bfH}_{\rm sp}\|_{\mathcal{H}_{\infty}}}
{\|\bfH_{\rm sp}\|_{\mathcal{H}_{\infty}}}$ for both reduced-order models by sampling the imaginary axis. 
These errors for the balanced truncation method and Algorithm~\ref{index2_IRKA} are, respectively, 
$3.3284 \times 10^{-6}$ and  $8.9663\times 10^{-6}$. Both reduced-order models are highly accurate. It is expected 
that the $\Hinf$-error in balanced truncation will be smaller than that in IRKA. 
While our method tries to minimize the $\Htwo$-norm, the balanced truncation method is  tailored towards reducing
the $\Hinf$-norm. Indeed, these numbers are further signs for the success of the
interpolatory-based model reduction method as it produces a very accurate model, almost matching the accuracy 
of the balanced truncation approach.
These observations are similar to those on IRKA  whose $\Hinf$-norm behavior was close to or even better in some 
cases than that of balanced truncation \cite{antoulas2010imr,gugercin2008hmr}.
   
\begin{figure}[ht]
   \centering
\includegraphics[scale=0.325]{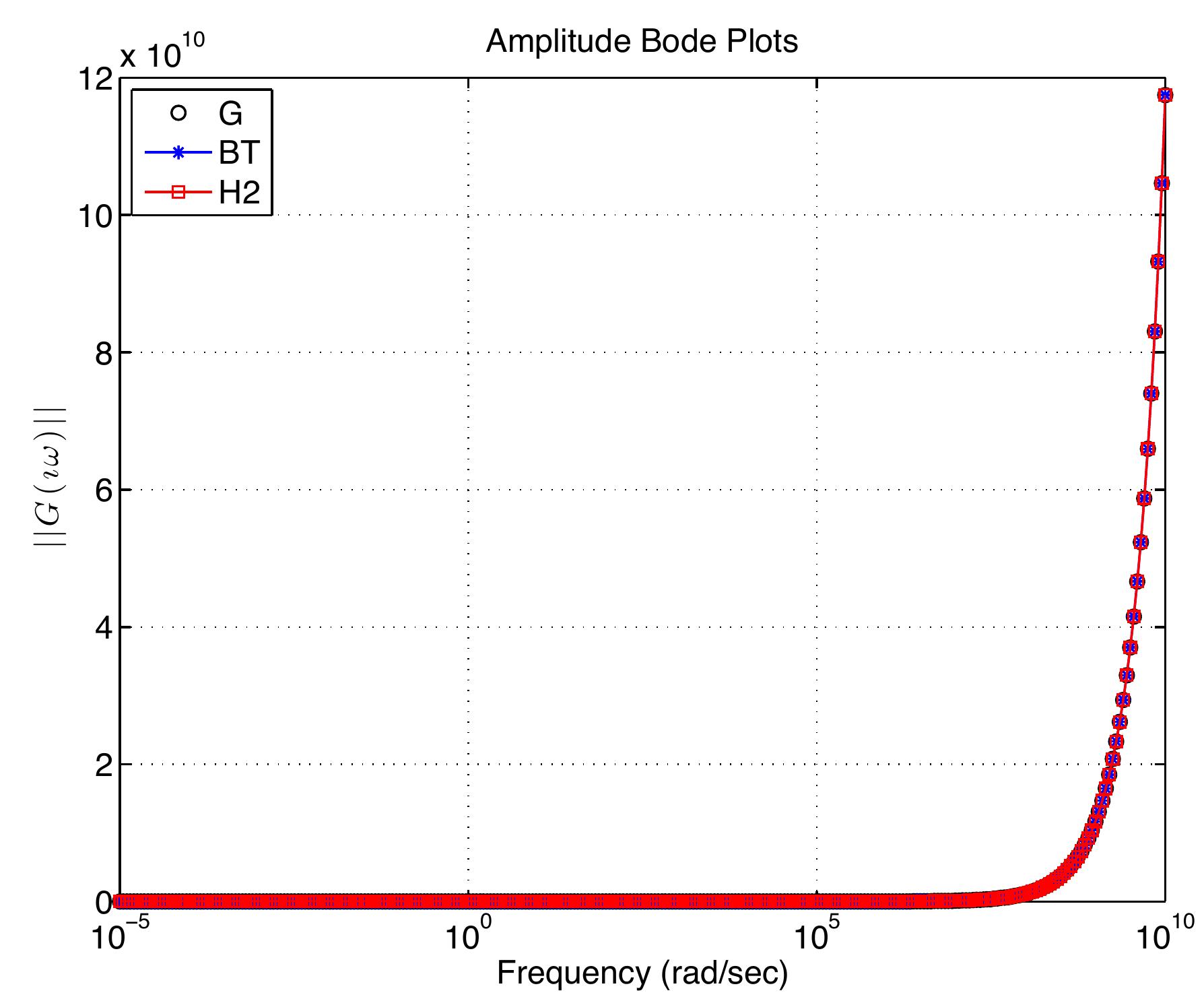}
\caption{\small {Oseen equation: amplitude Bode plots of the full and reduced models}}
   \label{fig:oseen_bode}
\end{figure}

To further illustrate the accuracy in the reduced-order model computed by Algorithm~\ref{index2_IRKA}, 
we display the time domain response plots resulting from two different input selections. 
In the left pane of Figure~\ref{fig:problem2_time1}, we plot the outputs for the input selections 
$\bfu_i(t) = \text{sin}(6it)$ for $i = 1, \ldots, 6$ (recall that the system has $6$ inputs). 
The figure illustrate a perfect match between the outputs of the full and reduced-order systems. 
Error in the outputs for the same input selection is given in the right pane of Figure~\ref{fig:problem2_time1}. 
Note the difference in the scale of the error plot compared to the actual output; the error is four orders of 
magnitude smaller. 
We repeat the same experiments with   ${\bfu}_i(t) = \text{sin}(it)$ for $i = 1, \ldots, 6$ and reach the same 
conclusions  as shown in Figure~\ref{fig:problem2_time2}.

\begin{figure}[ht]
   \centering
\includegraphics[scale=0.37]{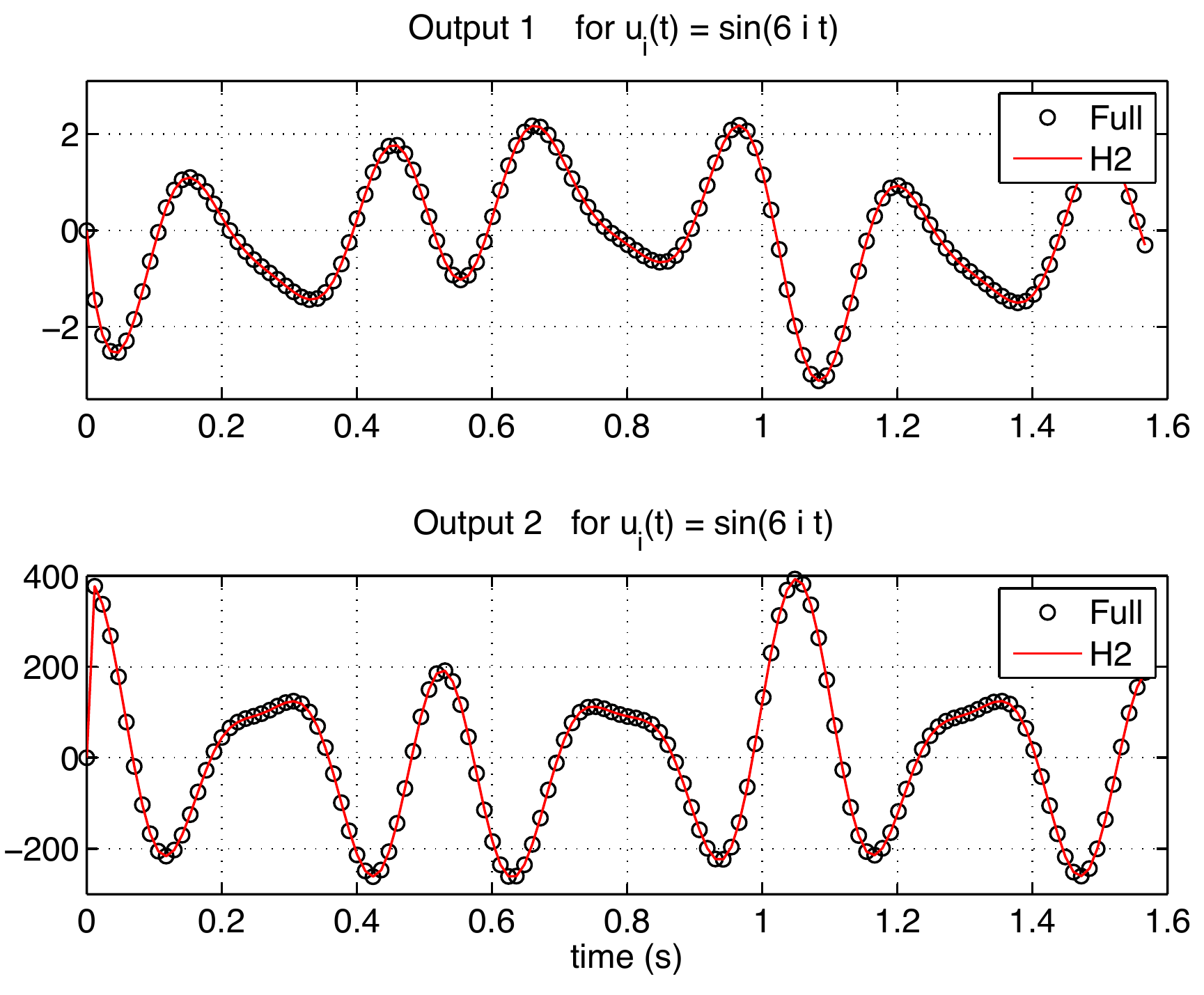}
\includegraphics[scale=0.37]{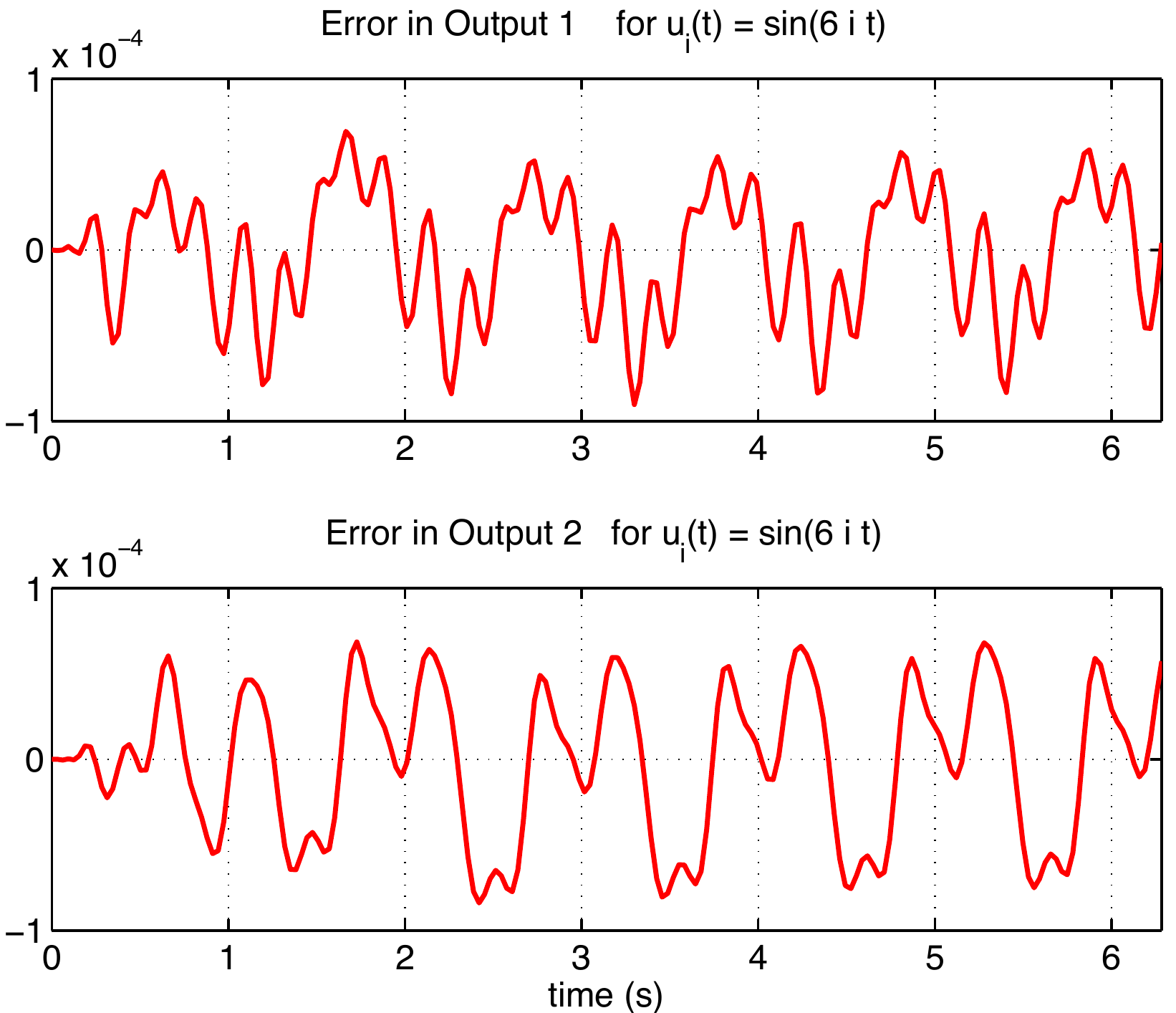}
\caption{\small {Oseen equation: (left) time domain response for ${\bfu}_i(t)=\sin(6it)$; 
          (right)  error in time domain response for ${\bfu}_i(t)=\sin(6it)$.}}
   \label{fig:problem2_time1}
\end{figure}

\begin{figure}[ht]
   \centering
\includegraphics[scale=0.37]{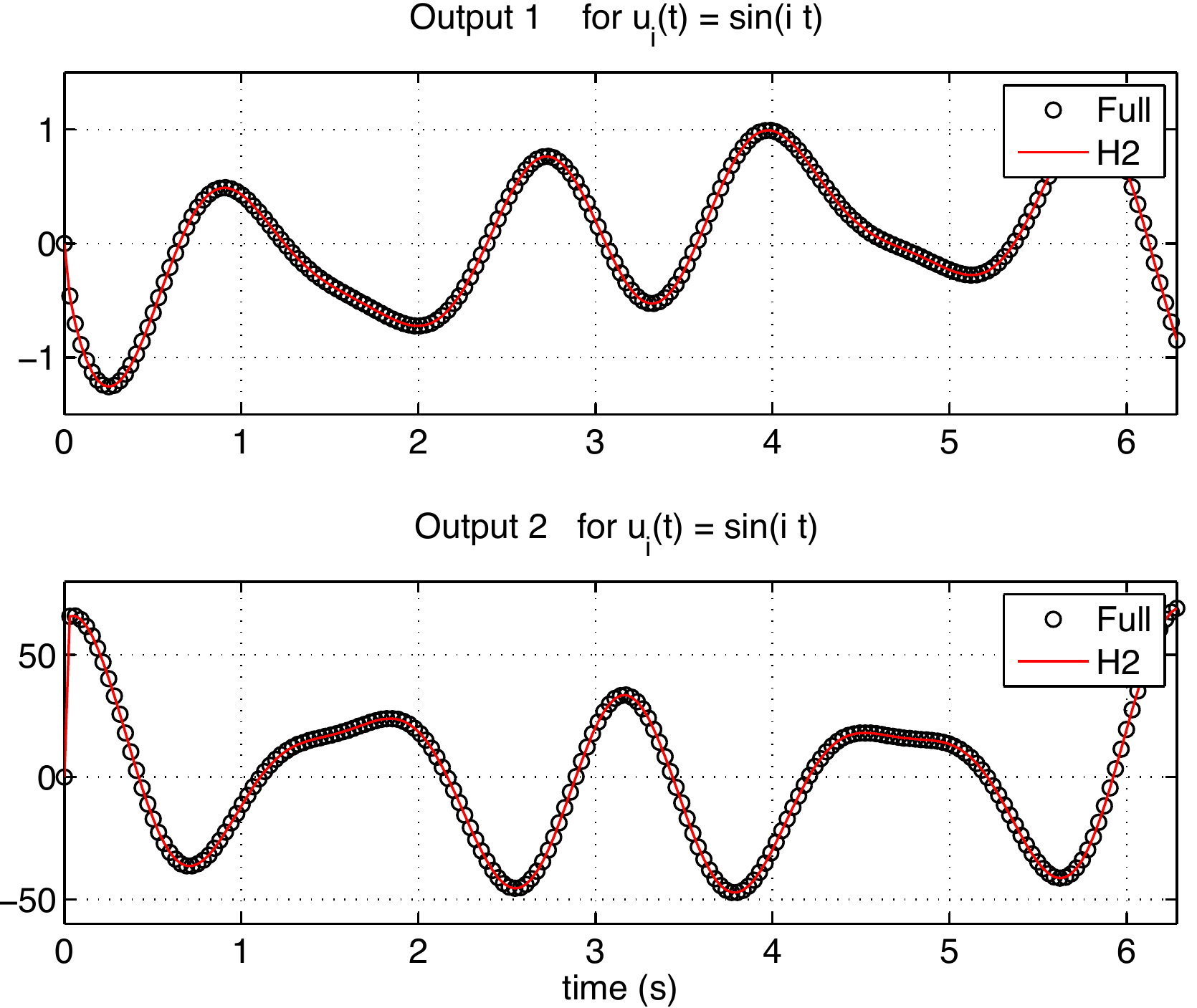}
\includegraphics[scale=0.37]{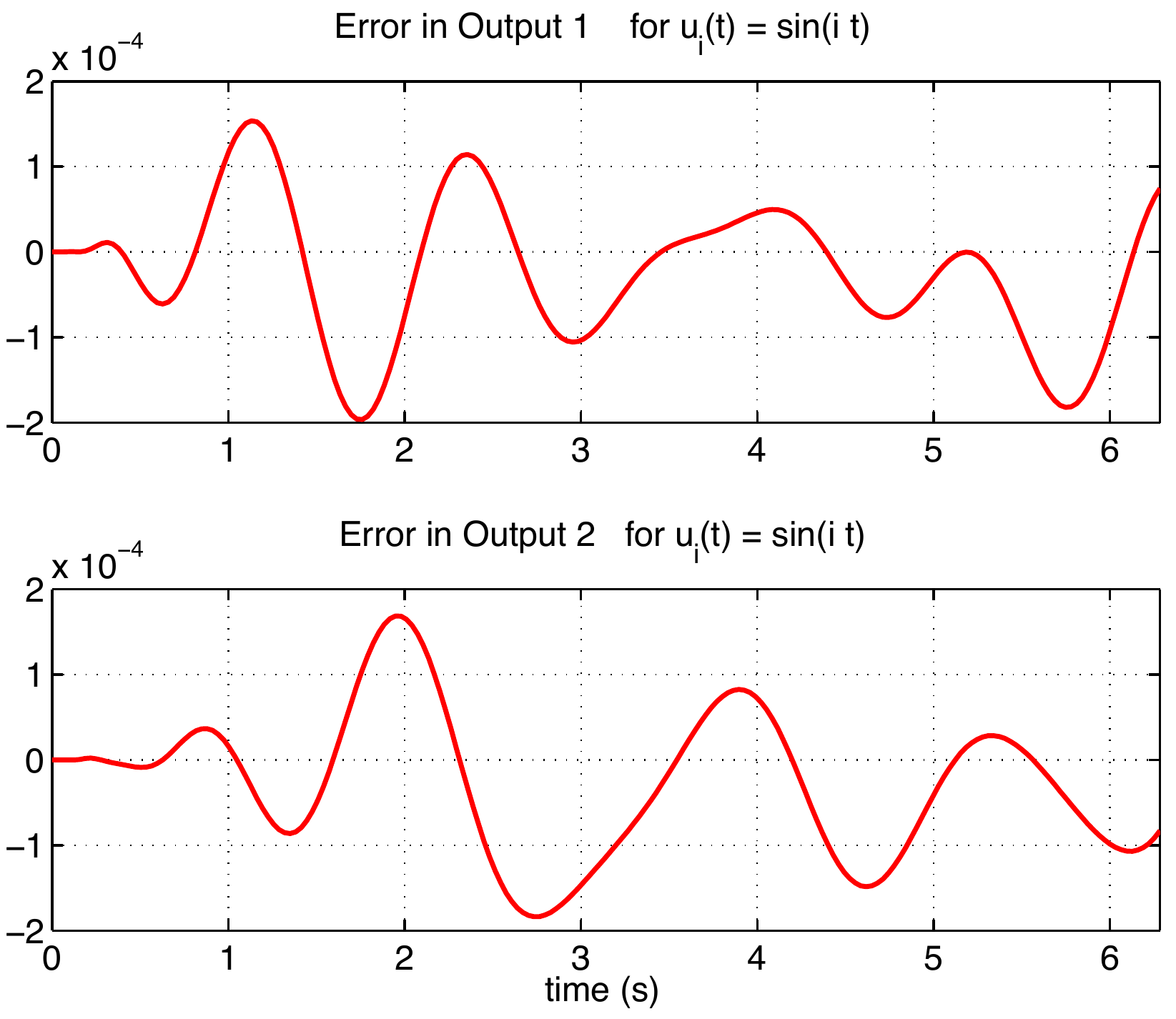}
\caption{\small {Oseen equation: (left) time domain response for $\bfu_i(t)=\sin( it)$; 
          (right)  error in time domain response for $\bfu_i(t)=\sin(it)$.}}
   \label{fig:problem2_time2}
\end{figure}

\section{Conclusions}
For interpolatory model reduction of descriptor systems,  we have  introduced subspace conditions that not only 
guarantee interpolation conditions but also automatically enforce matching the polynomial part of the transfer 
function, thus preventing the error grow unbounded. We have also extended the optimal $\Htwo$ interpolation point 
selection strategy to descriptor systems.
For the index-$1$ and index-$2$ descriptor systems, we have shown how to construct the reduced-order models 
without computing the deflating subspaces corresponding to the finite and infinite eigenvalues explicitly. 
 Several numerical examples have supported the theoretical discussion.

\section{Acknowledgements} 
The authors thank Prof. M. Heinkenschloss for providing the data and  the MATLAB files  for the numerical example 
of Section~\ref{sec:oseen_example}. The work of S.~Gugercin was supported in part by NSF through Grant DMS-0645347. 
The work of T.~Stykel was supported in part by the Research Network FROPT:
{\em Model Reduction Based Optimal Control for Field-Flow Fractionation}\, 
  funded by the German Federal Ministry of Education and Science (BMBF), 
  Grant~05M10WAB. 
 
 \bibliographystyle{plain}
\bibliography{interpolation_dae}

\end{document}